\documentclass[12pt]{amsart}
\pdfoutput=1

\usepackage{xcolor, amsmath, amsthm, amssymb, graphicx}
\usepackage{ifthen}
\usepackage{enumerate}

\usepackage{geometry}
\usepackage[bookmarks=false, colorlinks=true, urlcolor=blue, linkcolor=violet, citecolor=violet]{hyperref}

\definecolor{darkred}{rgb}{0.7,0,0} 
\newcommand{\defn}[1]{{\color{darkred}\emph{#1}}} 

\usepackage{tikz}
\usetikzlibrary{shapes,calc}

\newcommand*\circled[1]{\tikz[baseline=(char.base)]{
            \node[shape=circle,draw,inner sep=0.7pt] (char) {#1};}}

\newtheorem{theorem}{Theorem}[section]
\newtheorem{definition}[theorem]{Definition}
\newtheorem{example}[theorem]{Example}

\newtheorem{proposition}[theorem]{Proposition}
\newtheorem{remark}[theorem]{Remark}
\usepackage[utf8]{inputenc}

\begin{document}

\title{Cluster algebras and binary subwords}
\author{Rachel Bailey}
\address{Department of Mathematics, University of Connecticut, 
	Storrs, CT 06269-1009, USA}
\email{rachel.bailey@uconn.edu}
\author{Emily Gunawan}
\address{Department of Mathematics, University of Oklahoma, 
	Norman, OK 73019-3103, USA}
\email{egunawan@ou.edu}\thanks{
	Both authors were supported by the University of Connecticut. 
	E.G. was supported by the NSF grant DMS-1254567.}

\begin{abstract}
This paper establishes a connection between binary subwords 
and perfect matchings of a snake graph, an important tool in the theory of cluster algebras.  
Every binary expansion~$w$ can be associated to a piecewise-linear poset~$P$ and a snake graph~$G$.
We construct a tree structure called the antichain trie which is isomorphic to the trie of subwords introduced by Leroy, Rigo, and Stipulanti.
We then present bijections from the subwords of $w$ to the antichains of $P$ and to the perfect matchings of~$G$. 
	 
\end{abstract}

\maketitle

\section{Introduction}
\label{sec:poset}

A planar graph called the \defn{snake graph} appears naturally in the study of cluster algebras~\cite{FZ02}. 
An early version of the snake graph is a bipartite graph which is dual to a polygon triangulation and was studied by Propp et al. along with 
a combinatorial zoo of 
related models~\cite{Pro05}. 
Musiker, Schiffler, and Williams then used the perfect matchings of snake graphs to study positivity and bases of cluster algebras from surfaces~\cite{MSW11, MSW13}. 
The theory of abstract snake graph was developed further by {\c{C}anak\c{c}\i} and Schiffler~\cite{CS13}. 
The perfect matchings of snake graphs are connected to various mathematical objects,  
including 
lattice paths in a snake graph~\cite{Pro05, Cla20}, 
matchings of triangles and trails corresponding to these matchings~\cite{BCI74, Pro05, GMV16}, 
matchings of angles and minimal cuts~\cite{Yur18}, 
T-paths~\cite{Sch08,ST09}, 
submodules of a string module~\cite{CC06, MSW13, CaSi18}, 
0-1 sequences called globally compatible sequences (GCSs)~\cite{LLN17}, 
intervals in the weak order on the symmetric group determined by a Coxeter element~\cite{CaSi18}, 
 intervals in the Young lattice~\cite{Cla20}, 
 Conway--Coxeter frieze patterns~\cite{Pro05, Mor15}, 
continued fractions~\cite{CS18}, 
 Jones polynomials~\cite{LS17}, 
 Alexander polynomials~\cite{NT20}, 
 and HOMFLY polynomials~\cite{Yac19}. 
We add another item to this list by providing a connection between perfect matchings of a snake graph and base-2 expansions of positive integers.

In this paper, let a \defn{binary word} be a finite (possibly empty) sequence of letters on the alphabet $\{0,1 \}$  starting with $1$. 
Let a \defn{subword} of a binary word be a ``scattered" subsequence which is itself a binary word.

To every nonempty binary word $w=w_1 w_2 \dots w_d$ of length $d$ 
we associate (the Hasse diagram of) a \defn{piecewise-linear}\footnote{Other authors refer to these posets as `zig-zag-chain posets'~\cite{KMR18}, `path posets'~\cite{Yac19}, and `fences'~\cite{MSS20}.}  
partially ordered set (poset) $P$ as follows.
The elements of $P$ are labeled $P_1=1$, $\dots$, $P_d=d$, arranged from left to right in the Hasse diagram of~$P$, 
and there is an edge between $P_{i-1}$ and $P_{i}$.
For $i \geq 2$, if $w_i=1$ (respectively, if $w_i=0$) then the edge between $P_{i-1}$ and $P_i$ is of slope $1$ (respectively, $-1$), so that we have the 
covering relation $P_{i-1} \lessdot P_i$ (respectively, $P_{i-1} \gtrdot P_i$). See Fig.~\ref{fig:hasse101110left} (left).

One can also associate to $w=w_1 w_2 \dots w_d$
a quiver which is an orientation of the type $\mathbb{A}_d$ Dynkin diagram 
by assigning the orientation 
$P_{i-1} \to P_i$  if $w_i=1$ and $P_{i-1} \leftarrow P_i$ if $w_i=0$ for $i \geq 2$. See Fig.~\ref{fig:quiver101110right} (right).

An \defn{antichain} is a subset $A= \{A_1, A_2, \dots ,A_r\}$ of a poset 
 such that no two distinct elements in $A$ are comparable. 
 For example, the subsets $\{1,3,6\}$, $\{1,4\}$, and $\{2,6\}$ of the poset whose Hasse diagram is given in Fig.~\ref{fig:hasse101110left} (left) are antichains, while $\{2,4\}$ is not.

In~{\cite[Section~2]{LRS17}}, Leroy, Rigo, and Stipulanti introduce a specific construction of a prefix tree (called \defn{trie of subwords}) which is a binary tree that is convenient for counting distinct subwords occurring in a given word $w$. In this paper, we construct 
an analog (called the \defn{antichain trie}) to study the antichains of the poset $P$ corresponding to $w$.
We associate each node $v$ of the antichain trie to an antichain $A(v)$ of $P$ in such a way that moving from a node $v$ to its left child replaces $P_i$ in $A(v)$ with $P_{i+1}$ (where $i$ is the largest integer in $A(v)$) and moving from a node $v$ to its right child adds a new element $P_i$ to $A(v)$ (where $i$ is larger than every integer in $A(v)$). 

\begin{proposition}[Proposition \ref{prop:antichain_trie}] 
The nodes of the antichain trie are distinct antichains.
\end{proposition}

Next, we show that this antichain trie contains all antichains by giving a bijection between the subwords and the antichains. 

\begin{theorem}[Theorem \ref{thm:antichain_to_subword}]
Given a nonempty binary word $w$ and its corresponding piecewise-linear poset $P$, there is a bijection between the subwords of $w$ and the antichains of~$P$. 
\end{theorem}

It is known that one can associate a binary sequence of length $d-1$ to a snake graph with $d$ tiles (see Definition ~\ref{def:sign_function}). Given a binary word $w=w_1 w_2 \dots w_d$, we associate $(w_2, \dots, w_d)$ to a snake graph $G(w)$ and  
present a bijection from the subwords of $w$ to the perfect matchings of $G(w)$.

\begin{theorem}[Theorem \ref{thm:subwords_to_snake_graph_matchings}]
The subwords of a binary word $w$ are in bijection with the perfect matchings of its corresponding snake graph $G(w)$. 
\end{theorem}

The paper is organized as follows. 
In Section~\ref{sec:trie_of_subwords}, we recall the construction of the trie of subwords given by Leroy, Rigo, and Stipulanti.
In Section~\ref{sec:antichain_trie}, we introduce and describe the construction of the antichain trie.  
We define a map from the antichains to the subwords and prove that it is a bijection in Section~\ref{sec:bijection_between_antichains_and_subwords}.
In Section~\ref{sec:subwords_to_snake_graph_matchings}, 
 we give the necessary snake graph theory background and describe a bijection between subwords and perfect matchings.

 \begin{figure}[hbpt]

\begin{center}
\begin{tikzpicture}[xscale=1, yscale=1]
\def\posetedgecolor{blue}
\node(1) at (0,0) {{$1$}}; 
\node(2) at (1,-1) {{$2$}};
\node(3) at (2,0) {{$3$}};
\node(4) at (3,1) {{$4$}};
\node(5) at (4,2) {{$5$}};
\node(6) at (5,1) {{$6$}};

\node[\posetedgecolor] at (-0.7,0) {$\underline{1}$};
\draw[-, line width=1pt] (2) -- (1) node[\posetedgecolor,pos=0.7,below] {\small $0$}; 
\draw[-, line width=1pt]  (2) -- (3) node[\posetedgecolor,pos=0.2,above] {\small $1$};  
\draw[-, line width=1pt] (3) -- (4) node[\posetedgecolor,pos=0.2,above] {\small $1$};  
\draw[-, line width=1pt] (4) -- (5) node[\posetedgecolor,pos=0.2,above] {\small $1$}; 
\draw[-,line width=1pt] (6) -- (5) node[\posetedgecolor,pos=0.7,below] {\small $0$}; 
\end{tikzpicture}
\qquad
\begin{tikzpicture}[xscale=1, yscale=1]
\def\posetedgecolor{blue}
\node(1) at (0,0) {{$1$}}; 
\node(2) at (1,-1) {{$2$}};
\node(3) at (2,0) {{$3$}};
\node(4) at (3,1) {{$4$}};
\node(5) at (4,2) {{$5$}};
\node(6) at (5,1) {{$6$}};

\node[\posetedgecolor] at (-0.7,0) {$\underline{1}$};
\draw[->, line width=1pt, densely dotted] (2) -- (1) node[\posetedgecolor,pos=0.7,below] {\small $0$}; 
\draw[->, line width=1pt]  (2) -- (3) node[\posetedgecolor,pos=0.2,above] {\small $1$};  
\draw[->, line width=1pt] (3) -- (4) node[\posetedgecolor,pos=0.2,above] {\small $1$};  
\draw[->, line width=1pt] (4) -- (5) node[\posetedgecolor,pos=0.2,above] {\small $1$}; 
\draw[->,line width=1pt, densely dotted] (6) -- (5) node[\posetedgecolor,pos=0.7,below] {\small $0$}; 
\end{tikzpicture}
\caption{The Hasse diagram of the poset (left) and the type $\mathbb{A}_6$ quiver (right) associated to the binary word \protect\textcolor{blue}{$101110$}}
\label{fig:hasse101110left}
\label{fig:quiver101110right}
\end{center}

\begin{center}
\def\xfigscale{1.39}
\def\yfigscale{1.05}
\begin{tikzpicture}[xscale=\xfigscale,yscale=\yfigscale,>=latex]
\def\posetedgecolor{blue} 
\def\posetzeroedge{densely dotted} 
\node[circle,fill=black, inner sep=0pt,minimum size=7pt](empty) at (0,0) {};
\node[circle,fill=black, inner sep=0pt,minimum size=7pt](0-1) at (0,-1) {};
\node[circle,fill=black, inner sep=0pt,minimum size=7pt](0-2) at (0,-2) {};
\node[regular polygon,regular polygon sides=4,black, draw, fill=green
, line width=1pt, inner sep=0pt,minimum size=7pt](0-3) at (0,-3) {};
\node[circle,fill=black, inner sep=0pt,minimum size=7pt](0-4) at (0,-4) {};
\node[circle,fill=black, inner sep=0pt,minimum size=7pt](0-5) at (0,-5) {};
\node[diamond,draw, red, line width=2pt, inner sep=0pt,minimum size=7pt](0-6) at (0,-6) {};

\node[diamond,draw,  red, line width=2pt, inner sep=0pt,minimum size=7pt](1-3-6) at (1,-3) {}; 
\node[diamond, draw, red, line width=2pt, inner sep=0pt,minimum size=7pt](1-4-6) at (1,-4) {};
\node[diamond, draw, red, line width=2pt,inner sep=0pt,minimum size=7pt](1-5-6) at (1,-5) {};

\node[regular polygon,regular polygon sides=4,black, draw, fill=green
, line width=1pt, inner sep=0pt,minimum size=7pt](2-2-3) at (2,-2) {};
\node[circle,fill=black, inner sep=0pt,minimum size=7pt](2-3-4) at (2,-3) {};
\node[circle,fill=black, inner sep=0pt,minimum size=7pt](2-4-5) at (2,-4) {};
\node[diamond,draw, red, line width = 2pt, inner sep=0pt,minimum size=7pt](2-5-6) at (2,-5) {};

\node[diamond,draw, red,  line width=2pt,inner sep=0pt,minimum size=7pt](3-3-6) at (3,-3) {};
\node[diamond,draw, red, line width=2pt,  inner sep=0pt,minimum size=7pt](3-4-6) at (3,-4) {};

\draw[-] (empty) 
-- (0-1) node[\posetedgecolor,pos=0.5,right,left] {\small $1$};
\draw[\posetzeroedge] (0-1) 
 -- (0-2) node[\posetedgecolor,pos=0.5,right,left] {\small $0$};
 \draw (0-2)
 -- (0-3) node[\posetedgecolor,pos=0.5,right, left] {\small $1$};
 \draw (0-3)
 -- (0-4) node[\posetedgecolor,pos=0.5,right,left] {\small $1$};
 \draw (0-4)
 -- (0-5) node[\posetedgecolor,pos=0.5,right, left] {\small $1$};
 \draw[\posetzeroedge] (0-5) 
 -- (0-6) node[\posetedgecolor,pos=0.5,right,left] {\small $0$};

\draw[-,\posetzeroedge] (0-2) -- (1-3-6) node[\posetedgecolor,pos=0.5,right,left] {\small $0$};
\draw[-,\posetzeroedge] (0-3) -- (1-4-6) node[\posetedgecolor,pos=0.5,right,left] {\small $0$};
\draw[-,\posetzeroedge] (0-4) -- (1-5-6) node[\posetedgecolor,pos=0.5,right,left] {\small $0$};

\draw[-] (0-1) -- (2-2-3) node[\posetedgecolor,pos=0.5,right,left] {$1$};

\draw[-] (2-2-3) 
-- (2-3-4) node[\posetedgecolor,pos=0.5,right,left] {\small $1$}
-- (2-4-5) node[\posetedgecolor,pos=0.5,right,left] {\small $1$};
\draw[\posetzeroedge] (2-4-5)
-- (2-5-6) node[\posetedgecolor,pos=0.5,right,left] {\small $0$};

\draw[\posetzeroedge] (2-2-3) -- (3-3-6) node[\posetedgecolor,pos=0.5,right,left] {\small $0$};

\draw[\posetzeroedge] (2-3-4) -- (3-4-6) node[\posetedgecolor,pos=0.5,right,left] {\small $0$};

\draw[thick, dotted] (0,-6) circle (4mm);      
\draw[thick,dotted] (-.5,-6.5)..controls (-.5,-2) and  (0, -2) ..(.75,-3.25);

\draw[thick,dotted] (.75,-3.25)..controls (1,-3.5) and  (1.5, -3.5) ..(1.75,-6.5);

\node at (0,-6.9) {$T_3$};

\node at (1.8,-6.9) {$T_2$};

\end{tikzpicture}
~
\begin{tikzpicture}[xscale=\xfigscale,yscale=\yfigscale]
\def\posetedgecolor{blue} 
\node(empty) at (0,0) {$0$};
\node(0-1) at (0,-1) {$1$};
\node(0-2) at (0,-2) {$2$};
\node(0-3) at (0,-3) {$3$};
\node(0-4) at (0,-4) {$4$};
\node(0-5) at (0,-5) {$5$};
\node(0-6) at (0,-6) {$6$};

\node(1-3-6) at (1,-3) {$6$};
\node(1-4-6) at (1,-4) {$6$};
\node(1-5-6) at (1,-5) {$6$};

\node(2-2-3) at (2,-2) {$3$};
\node(2-3-4) at (2,-3) {$4$};
\node(2-4-5) at (2,-4) {$5$};
\node(2-5-6) at (2,-5) {$6$};

\node(3-3-6) at (3,-3) {$6$};
\node(3-4-6) at (3,-4) {$6$};

\draw[-] (empty) 
-- (0-1) node[\posetedgecolor,pos=0.5,right,left] {}
 -- (0-2) node[\posetedgecolor,pos=0.2,below] {} 
 -- (0-3) node[\posetedgecolor,pos=0.5,above] {}
 -- (0-4) node[\posetedgecolor,pos=0.5,left] {}
 -- (0-5) node[\posetedgecolor,pos=0.5,right=1pt] {}
 -- (0-6) node[\posetedgecolor,pos=0.5,left=1mm] {};

\draw[-] (0-2) -- (1-3-6) node[\posetedgecolor,pos=0.5,right,left] {};
\draw[-] (0-3) -- (1-4-6) node[\posetedgecolor,pos=0.5,right,left] {};
\draw[-] (0-4) -- (1-5-6) node[\posetedgecolor,pos=0.5,right,left] {};

\draw[-] (0-1) -- (2-2-3) node[\posetedgecolor,pos=0.5,right,left] {};

\draw[-] (2-2-3) 
-- (2-3-4) node[\posetedgecolor,pos=0.5,right,left] {}
-- (2-4-5) node[\posetedgecolor,pos=0.5,right,left] {}
-- (2-5-6) node[\posetedgecolor,pos=0.5,right,left] {};

\draw[-] (2-2-3) -- (3-3-6) node[\posetedgecolor,pos=0.5,right,left] {};

\draw[-] (2-3-4) -- (3-4-6) node[\posetedgecolor,pos=0.5,right,left] {};

\draw[thick, dotted] (0,-6) circle (4.5mm);      
\draw[thick,dotted] (-.55,-6.5)..controls (-.5,-2) and  (0, -2) ..(.75,-3.25);

\draw[thick,dotted] (.75,-3.25)..controls (1,-3.5) and  (1.5, -3.5) ..(1.75,-6.5);

\node at (0,-6.9) {$\Gamma_2$};

\node at (1.8,-6.9) {$\Gamma_1$};

\end{tikzpicture}
~
\begin{tikzpicture}[xscale=\xfigscale,yscale=\yfigscale,>=latex]
\def\posetedgecolor{blue} 
\node(empty) at (0,0) {$\emptyset$};
\node(0-1) at (0,-1) {$\{1\}$};
\node(0-2) at (0,-2) {$\{2\}$};
\node(0-3) at (0,-3) {$\{3\}$};
\node(0-4) at (0,-4) {$\{4\}$};
\node(0-5) at (0,-5) {$\{5\}$};
\node(0-6) at (0,-6) {$\{6\}$};

\node(1-3-6) at (1,-3) {$\{2,6\}$};
\node(1-4-6) at (1,-4) {$\{3,6\}$};
\node(1-5-6) at (1,-5) {$\{4,6\}$};

\node(2-2-3) at (2,-2) {$\{1,3\}$};
\node(2-3-4) at (2,-3) {$\{1,4\}$};
\node(2-4-5) at (2,-4) {$\{1,5\}$};
\node(2-5-6) at (2,-5) {$\{1,6\}$};

\node(3-3-6) at (3.2,-3) {$\{1,3,6\}$};
\node(3-4-6) at (3.2,-4) {$\{1,4,6\}$};

\draw[-] (empty) 
-- (0-1) node[\posetedgecolor,pos=0.5,right,left] {}
 -- (0-2) node[\posetedgecolor,pos=0.2,below] {} 
 -- (0-3) node[\posetedgecolor,pos=0.5,above] {}
 -- (0-4) node[\posetedgecolor,pos=0.5,left] {}
 -- (0-5) node[\posetedgecolor,pos=0.5,right=1pt] {}
 -- (0-6) node[\posetedgecolor,pos=0.5,left=1mm] {};

\draw[-] (0-2) -- (1-3-6) node[\posetedgecolor,pos=0.5,right,left] {};
\draw[-] (0-3) -- (1-4-6) node[\posetedgecolor,pos=0.5,right,left] {};
\draw[-] (0-4) -- (1-5-6) node[\posetedgecolor,pos=0.5,right,left] {};

\draw[-] (0-1) -- (2-2-3) node[\posetedgecolor,pos=0.5,right,left] {};

\draw[-] (2-2-3) 
-- (2-3-4) node[\posetedgecolor,pos=0.5,right,left] {}
-- (2-4-5) node[\posetedgecolor,pos=0.5,right,left] {}
-- (2-5-6) node[\posetedgecolor,pos=0.5,right,left] {};

\draw[-] (2-2-3) -- (3-3-6) node[\posetedgecolor,pos=0.5,right,left] {};

\draw[-] (2-3-4) -- (3-4-6) node[\posetedgecolor,pos=0.5,right,left] {};

\node at (0,-6.9) {};
\end{tikzpicture}
\caption{The trie of subwords (left) of \textcolor{blue}{101110},  the antichain trie of the corresponding poset $P$ (center) and the antichains of $P$ (right)}
\label{fig:trie101110} 
\label{fig:ex:antichaintrie_101110} 
\label{fig:trie_antichainsof101110} 
\end{center}
\end{figure}

\section{Trie of subwords}
\label{sec:trie_of_subwords}

Let $w = w_1\dots w_d$ be a nonempty binary word and  
consider the \defn{trie of (distinct) subwords} of $w$, denoted by $\mathcal{T}$. 
It is a tree with the root denoted by $\epsilon$. 
If $u$ and $ua$ are two subwords of $w$ with $a$ being a one-letter subword, then $ua$ is a child of $u$. This trie is also called a \defn{prefix tree} because all successors of a node have a common prefix. 
Note that, since $w$ is a binary word, the trie is a binary tree. 
For the rest of the section, we describe the specific construction of $\mathcal{T}$ which is given in~{\cite[Section~2]{LRS17}}.

Factor $w$ into consecutive maximal blocks of $1$'s and blocks of $0$'s such that 
\[
w=\underbrace{1^{n_1}}_{u_1} \underbrace{0^{n_2}}_{u_2} \underbrace{1^{n_3}}_{u_3} \underbrace{0^{n_4}}_{u_4} \cdots \underbrace{1^{n_{2j-1}}}_{u_{2j-1}} \underbrace{0^{n_{2j}}}_{u_{2j}}
\]
with $j\geq 1$, $n_1, \dots, n_{2j-1} \geq 1$ and $n_{2j} \geq 0$.
Let $M$ be such that $w=u_1 u_2 \dots u_M$
where $u_M$ is the last non-empty block of $0$'s or $1$'s.

To construct the trie $\mathcal{T}$, begin with a vertical linear tree $T_w$ with 
nodes 
$v_0$,  
$\dots$, $v_d$. 
Let $T_w$ be rooted at $\epsilon = v_0$ and let node $v_i$ be the left child of node $v_{i-1}$ for all $i=1,\dots,d$. 
Label the edges 
of $T_w$ 
with the letters of $w$ such that the edge between nodes $v_{i-1}$ and $v_{i}$ is labeled $w_{i}$. We identify each node $v$ by the path of edge labels from $\epsilon$ to $v$. 

Starting from the bottom of the vertical linear tree $T_w$, we define a tree $T_l$ for every $l \in \{M-1, \dots, 2, 1\}$. Each tree is rooted at the node $u_1 \dots u_l1$ if $l$ is even and $u_1 \dots u_l0$ if $l$ is odd. 
First, let $T_{M-1}$ be the (linear) subtree of $T_w$ consisting of the last $n_M$ nodes.

We then attach a copy of $T_{M-1}$ to each node (on the vertical tree $T_w$) of the form
\[
\begin{cases}
  u_1u_2 \dots u_{M-2}1^j, & \text{if $u_{M-1}$ is a block of 1s}\\
  u_1u_2 \dots u_{M-2}0^j, & \text{if $u_{M-1}$ is a block of 0s}
\end{cases}
~~~~~~~ \text{ for $j \in \{0,1 \dots, n_{M-1}-1\}$.}
\]
Let the root of each copy of $T_{M-1}$ be the right child of the node of $T_w$ that this root is attached to. This results in a (non-linear) tree $T_w'$ that is larger than $T_w$. 

Let $T_{M-2}$ be the 
subtree of this larger tree $T_w'$ such that its root is   
 $u_1 \dots u_{M-2}1$ if $M-2$ is even and $u_1 \dots u_{M-2}0$ if $M-2$ is odd and $T_{M-2}$ contains all the descendants of this root.  
Then attach a copy of $T_{M-2}$ to each node of the form
\[
\begin{cases}
 u_1 u_2 \dots  u_{M-3}1^j, & \text{if $u_{M-2}$ is a block of 1s}\\
 u_1 u_2 \dots u_{M-3}0^j, & \text{if $u_{M-2}$ is a block of 0s}
 \end{cases}
~~~~~~~ \text{ for $ j\in \{0,1,...,n_{M-2}-1\}$.}
\]
Again, let the root of each copy of $T_{M-2}$ be the right child of the node of $T_w$ that this root is attached to. 

Let $T_{M-3}$ be the subtree of this larger tree such that it is rooted at  
 $u_1$$\dots$$u_{M-3}1$ (resp., $u_1 \dots u_{M-3}0$) if $M-3$ is even (resp., odd) and $T_{M-3}$ contains all descendants of this root.

Continue as such until after we attach a copy of $T_2$. If $n_1=1$ then no copy of $T_1$ is added (as in Fig.~\ref{fig:trie101110} and Fig.~\ref{fig:subwordtrie_10010111} (left)). If $n_1 > 1$, then a copy of $T_1$ is added to each node of the form $1^j, j\in \{0, 1,\dots ,n_1 -1\}$ (as in~\cite[Example~8]{LRS17}).

When $T_l$ is copied, keep its respective edge labels. The new edge connecting a copy of $T_l$ to the original vertical linear tree $T_w$ has the same label as the edge (of $T_w$) above the root of the original copy of $T_l$.

\begin{example}
\label{example:fig:trie101110}
Fig.~\ref{fig:trie101110} (left) shows the complete trie of subwords for the word $101110$. Since $w = \underbrace{1^1}_{u_1}\underbrace{0^1}_{u_2}\underbrace{1^3}_{u_3}\underbrace{0^1}_{u_4}$, we have $M = 4$.  
The subtree $T_3$ is the sole diamond node on $T_w$ because $T_3$ is rooted at the node $u_1u_2u_30$. We then attach a copy of $T_3$ to the nodes $ u_1u_21^j, j\in\{0,1,2\}$. The root of $T_2$ is the node $u_1u_21$ (the square node) and we attach a copy of $T_2$ to the node $u_10^j, j \in \{0\}$. Lastly, because $n_1 = 1$, no copy of $T_1$ is added.
See also Fig.~\ref{fig:subwordtrie_10010111} (left) and~\cite[Figs.~3-4]{LRS17}.

\end{example}\label{ex:trieofsubwords_101110}

\begin{figure}[hbpt]
\begin{center}
\begin{tikzpicture}[xscale=1,yscale=0.5,>=latex]
\def\posetedgecolor{blue}
\node(1) at (0,0) {{$1$}}; 
\node(2) at (1,-1) {{$2$}}; 
\node(3) at (2,-2) {{$3$}}; 
\node(4) at (3,-1) {{$4$}}; 
\node(5) at (4,-2) {{$5$}}; 
\node(6) at (5,-1) {{$6$}}; 
\node(7) at (6,0) {{$7$}}; 
\node(8) at (7,1) {{$8$}}; 

\node[\posetedgecolor] at (-0.7,0) {$\underline{1}$};
\draw[-, line width = 1pt] (2) -- (1) node[\posetedgecolor,pos=0.7,below] {\small $0$}; 
\draw[-, line width = 1pt]  (3) -- (2) node[\posetedgecolor,pos=0.7,below] {\small $0$}; ; 
\draw[-, line width = 1pt] (3) -- (4) node[\posetedgecolor,pos=0.2,above] {\small $1$}; ; 
\draw[-, line width = 1pt] (5) -- (4) node[\posetedgecolor,pos=0.7,below] {\small $0$}; ;
\draw[-,line width=1pt] (5) -- (6) node[\posetedgecolor,pos=0.2,above] {\small $1$}; ; 
\draw[-, line width = 1pt] (6) -- (7) node[\posetedgecolor,pos=0.2,above] {\small $1$}; ; 
\draw[-, line width = 1pt] (7) -- (8) node[\posetedgecolor,pos=0.2,above] {\small $1$}; ; 
\end{tikzpicture}

\caption{The Hasse diagram of the poset corresponding to the word \textcolor{blue}{10010111}} \label{hasse_10010111}

\def\xfigscale{1.1}
\def\yfigscale{0.99}

\begin{tikzpicture}[xscale=\xfigscale,yscale=\yfigscale]
\def\posetedgecolor{blue} 
\def\posetzeroedge{densely dotted}

\node[circle,fill=black, inner sep=0pt,minimum size=7pt](empty) at (0,0) {};
\node[circle,fill=black, inner sep=0pt,minimum size=7pt](0-1) at (0,-1) {};
\node[circle,fill=black, inner sep=0pt,minimum size=7pt](0-2) at (0,-2) {};
\node[circle,fill=black, inner sep=0pt,minimum size=7pt](0-3) at (0,-3) {};
\node[star,star points=5,black, draw,fill=blue,  line width=2pt, inner sep=0pt,minimum size=7pt](0-4) at (0,-4) {};
\node[regular polygon,regular polygon sides=4,black, draw, fill=green, line width=1pt, inner sep=0pt,minimum size=7pt](0-5) at (0,-5) {};
\node[diamond,draw, red, line width=2pt, inner sep=0pt,minimum size=7pt](0-6) at (0,-6) {};
\node[circle,fill=black, inner sep=0pt,minimum size=7pt](0-7) at (0,-7) {};
\node[circle,fill=black, inner sep=0pt,minimum size=7pt](0-8) at (0,-8) {};

\node[diamond,draw, red, line width=2pt, inner sep=0pt,minimum size=7pt](1-5-6) at (1,-5) {};
\node[circle,fill=black, inner sep=0pt,minimum size=7pt](1-6-7) at (1,-6) {};
\node[circle,fill=black, inner sep=0pt,minimum size=7pt](1-7-8) at (1,-7) {};

\node[regular polygon,regular polygon sides=4,black, draw, fill=green, line width=1pt, inner sep=0pt,minimum size=7pt](2-4-5) at (2,-4) {};
\node[diamond,draw,red, line width=2pt, inner sep=0pt,minimum size=7pt](2-5-6) at (2,-5) {};
\node[circle,fill=black, inner sep=0pt,minimum size=7pt](2-6-7) at (2,-6) {};
\node[circle,fill=black, inner sep=0pt,minimum size=7pt](2-7-8) at (2,-7) {};

\node[star,star points=5,black, draw,fill=blue, line width=2pt, inner sep=0pt,minimum size=7pt](3-3-4) at (3,-3) {};
\node[regular polygon,regular polygon sides=4, draw, black, fill=green, line width=1pt, inner sep=0pt,minimum size=7pt](3-4-5) at (3,-4) {};
\node[diamond,draw, red, line width=2pt, inner sep=0pt,minimum size=7pt](3-5-6) at (3,-5) {};
\node[circle,fill=black, inner sep=0pt,minimum size=7pt](3-6-7) at (3,-6) {};
\node[circle,fill=black, inner sep=0pt,minimum size=7pt](3-7-8) at (3,-7) {};

\node[diamond,draw,red, line width=2pt,inner sep=0pt,minimum size=7pt](4-4-6) at (4,-4) {};
\node[circle,fill=black, inner sep=0pt,minimum size=7pt](4-5-7) at (4,-5) {};
\node[circle,fill=black, inner sep=0pt,minimum size=7pt](4-6-8) at (4,-6) {};

\node[star,star points=5,black, draw,fill=blue, line width=2pt, inner sep=0pt,minimum size=7pt](5-2-4) at (5,-2) {};
\node[regular polygon,regular polygon sides=4,black, draw, fill=green, line width=1pt, inner sep=0pt,minimum size=7pt](5-3-5) at (5,-3) {};
\node[diamond,draw,red, line width=2pt, inner sep=0pt,minimum size=7pt](5-4-6) at (5,-4) {};
\node[circle,fill=black, inner sep=0pt,minimum size=7pt](5-5-7) at (5,-5) {};
\node[circle,fill=black, inner sep=0pt,minimum size=7pt](5-6-8) at (5,-6) {};

\node[diamond,draw,red, line width=2pt,inner sep=0pt,minimum size=7pt](6-3-6) at (6,-3) {};
\node[circle,fill=black, inner sep=0pt,minimum size=7pt](6-4-7) at (6,-4) {};
\node[circle,fill=black, inner sep=0pt,minimum size=7pt](6-5-8) at (6,-5) {};

\draw[-] (empty) 
-- (0-1) node[\posetedgecolor,pos=0.5,right,left] {$1$};
\draw[\posetzeroedge](0-1)
 -- (0-2) node[\posetedgecolor, pos=0.5,right,left] {$0$};
 \draw[\posetzeroedge](0-2)
 --(0-3) node[\posetedgecolor,pos=0.5,right,left] {$0$};
 \draw (0-3)
 -- (0-4) node[\posetedgecolor,pos=0.5,right,left] {$1$};
 \draw[\posetzeroedge](0-4)
 -- (0-5) node[\posetedgecolor,pos=0.5,right,left] {$0$};
 \draw (0-5)
 -- (0-6) node[\posetedgecolor,pos=0.5,right,left] {$1$};
 \draw(0-6)
 -- (0-7) node[\posetedgecolor,pos=0.5,right,left] {$1$};
 \draw(0-7)
 -- (0-8) node[\posetedgecolor,pos=0.5,right,left] {$1$};

\draw[-] 
-- (1-5-6) node[\posetedgecolor,pos=0.5,right,left] {}
 -- (1-6-7) node[\posetedgecolor,pos=0.5,right,left] {$1$} 
 -- (1-7-8) node[\posetedgecolor,pos=0.5,right,left] {$1$};
 
 \draw[-] 
-- (2-4-5) node[\posetedgecolor,pos=0.5,right,left] {}
 -- (2-5-6) node[\posetedgecolor,pos=0.5,right,left] {$1$} 
 -- (2-6-7) node[\posetedgecolor,pos=0.5,right,left] {$1$}
 --(2-7-8) node[\posetedgecolor,pos=0.5,right,left] {$1$};
 
  \draw[-] 
-- (3-3-4) node[\posetedgecolor,pos=0.5,right,left] {};
\draw[\posetzeroedge](3-3-4)
 -- (3-4-5) node[\posetedgecolor,pos=0.5,right,left] {$0$};
 \draw (3-4-5)
 -- (3-5-6) node[\posetedgecolor,pos=0.5,right,left] {$1$};
 \draw(3-5-6)
 --(3-6-7) node[\posetedgecolor,pos=0.5,right,left] {$1$};
 \draw(3-6-7)
 --(3-7-8) node[\posetedgecolor,pos=0.5,right,left]{$1$};

 \draw[-] 
-- (4-4-6) node[\posetedgecolor,pos=0.5,right,left] {}
 -- (4-5-7) node[\posetedgecolor,pos=0.5,right,left] {$1$} 
 -- (4-6-8) node[\posetedgecolor,pos=0.5,right,left] {$1$};

 \draw[-] 
-- (5-2-4) node[\posetedgecolor,pos=0.5,right,left] {};
\draw[\posetzeroedge](5-2-4)
 -- (5-3-5) node[\posetedgecolor,pos=0.5,right,left] {$0$} ;
 \draw(5-3-5)
 -- (5-4-6) node[\posetedgecolor,pos=0.5,right,left] {$1$};
 \draw(5-4-6)
 --(5-5-7) node[\posetedgecolor,pos=0.5,right,left] {$1$};
 \draw(5-5-7)
 --(5-6-8) node[\posetedgecolor,pos=0.5,right,left] {$1$};
 
  \draw[-] 
-- (6-3-6) node[\posetedgecolor,pos=0.5,right,left] {}
 -- (6-4-7) node[\posetedgecolor,pos=0.5,right,left] {$1$} 
 -- (6-5-8) node[\posetedgecolor,pos=0.5,right,left] {$1$};

\draw[-] (0-1) -- (5-2-4) node[\posetedgecolor,pos=0.5,above] {$1$};
\draw[-] (0-2) -- (3-3-4) node[\posetedgecolor,pos=0.5,above] {$1$};
\draw[\posetzeroedge] (0-3) -- (2-4-5) node[\posetedgecolor,pos=0.5,above] {$0$};

\draw[-] (0-4) -- (1-5-6) node[\posetedgecolor,pos=0.5,above] {$1$};

\draw[-] (5-2-4) -- (6-3-6) node[\posetedgecolor,pos=0.5,above] {$1$};
\draw[-] (3-3-4) -- (4-4-6) node[\posetedgecolor,pos=0.5,above] {$1$};

\draw[thick,dotted] (-.3,-8.2)..controls (-.5,-4.5) and  (.5, -4.5) ..(.3,-8.2);

\draw[thick,dotted] (-1,-8.2)..controls (-.5,-3.25) and  (.5, -3.25) ..(.75,-8.2);

\draw[thick,dotted] (-1.5,-8.1)..controls (-.75,-3) and  (0, -3) ..(1, -4.25);

\draw[thick,dotted] (1,-4.25)..controls (1.25,-4.5) and  (1.5, -4.75) ..(1.75,-8.1);

\node(a) at (.25,-8.4) {$T_4$};
\node at (.8,-8.4) {$T_3$};
\node at (1.8,-8.4) {$T_2$};

\end{tikzpicture}~~~~
\begin{tikzpicture}[xscale=\xfigscale,yscale=\yfigscale]
\def\posetedgecolor{blue} 
\node(empty) at (0,0) {$0$};
\node(0-1) at (0,-1) {$1$};
\node(0-2) at (0,-2) {$2$};
\node(0-3) at (0,-3) {$3$};
\node(0-4) at (0,-4) {$4$};
\node(0-5) at (0,-5) {$5$};
\node(0-6) at (0,-6) {$6$};
\node(0-7) at (0,-7) {$7$};
\node(0-8) at (0,-8) {$8$};

\node(1-5-6) at (1,-5) {$6$};
\node(1-6-7) at (1,-6) {$7$};
\node(1-7-8) at (1,-7) {$8$};

\node(2-4-5) at (2,-4) {$5$};
\node(2-5-6) at (2,-5) {$6$};
\node(2-6-7) at (2,-6) {$7$};
\node(2-7-8) at (2,-7) {$8$};

\node(3-3-4) at (3,-3) {$4$};
\node(3-4-5) at (3,-4) {$5$};
\node(3-5-6) at (3,-5) {$6$};
\node(3-6-7) at (3,-6) {$7$};
\node(3-7-8) at (3,-7) {$8$};

\node(4-4-6) at (4,-4) {$6$};
\node(4-5-7) at (4,-5) {$7$};
\node(4-6-8) at (4,-6) {$8$};

\node(5-2-4) at (5,-2) {$4$};
\node(5-3-5) at (5,-3) {$5$};
\node(5-4-6) at (5,-4) {$6$};
\node(5-5-7) at (5,-5) {$7$};
\node(5-6-8) at (5,-6) {$8$};

\node(6-3-6) at (6,-3) {$6$};
\node(6-4-7) at (6,-4) {$7$};
\node(6-5-8) at (6,-5) {$8$};

\draw[-] (empty) 
-- (0-1) node[\posetedgecolor,pos=0.5,right,left] {}
 -- (0-2) node[\posetedgecolor,pos=0.2,below] {} 
 -- (0-3) node[\posetedgecolor,pos=0.5,above] {}
 -- (0-4) node[\posetedgecolor,pos=0.5,left] {}
 -- (0-5) node[\posetedgecolor,pos=0.5,right=1pt] {}
 -- (0-6) node[\posetedgecolor,pos=0.5,left=1mm] {}
 -- (0-7) node[\posetedgecolor,pos=0.5,left=1mm] {}
 -- (0-8) node[\posetedgecolor,pos=0.5,left=1mm] {};

\draw[-] 
-- (1-5-6) node[\posetedgecolor,pos=0.5,right,left] {}
 -- (1-6-7) node[\posetedgecolor,pos=0.2,below] {} 
 -- (1-7-8) node[\posetedgecolor,pos=0.5,above] {};
 
 \draw[-] 
-- (2-4-5) node[\posetedgecolor,pos=0.5,right,left] {}
 -- (2-5-6) node[\posetedgecolor,pos=0.2,below] {} 
 -- (2-6-7) node[\posetedgecolor,pos=0.5,above] {}
 --(2-7-8) node[\posetedgecolor,pos=0.5,above] {};
 
  \draw[-] 
-- (3-3-4) node[\posetedgecolor,pos=0.5,right,left] {}
 -- (3-4-5) node[\posetedgecolor,pos=0.2,below] {} 
 -- (3-5-6) node[\posetedgecolor,pos=0.5,above] {}
 --(3-6-7) node[\posetedgecolor,pos=0.5,above] {}
 --(3-7-8) node[\posetedgecolor,pos=0.5,above]{};

 \draw[-] 
-- (4-4-6) node[\posetedgecolor,pos=0.5,right,left] {}
 -- (4-5-7) node[\posetedgecolor,pos=0.2,below] {} 
 -- (4-6-8) node[\posetedgecolor,pos=0.5,above] {};

 \draw[-] 
-- (5-2-4) node[\posetedgecolor,pos=0.5,right,left] {}
 -- (5-3-5) node[\posetedgecolor,pos=0.2,below] {} 
 -- (5-4-6) node[\posetedgecolor,pos=0.5,above] {}
 --(5-5-7) node[\posetedgecolor,pos=0.5,above] {}
 --(5-6-8) node[\posetedgecolor,pos=0.5,above] {};
 
  \draw[-] 
-- (6-3-6) node[\posetedgecolor,pos=0.5,right,left] {}
 -- (6-4-7) node[\posetedgecolor,pos=0.2,below] {} 
 -- (6-5-8) node[\posetedgecolor,pos=0.5,above] {};

\draw[-] (0-1) -- (5-2-4) node[\posetedgecolor,pos=0.5,right,left] {};
\draw[-] (0-2) -- (3-3-4) node[\posetedgecolor,pos=0.5,right,left] {$$};
\draw[-] (0-3) -- (2-4-5) node[\posetedgecolor,pos=0.5,right,left] {$$};
\draw[-] (0-4) -- (1-5-6) node[\posetedgecolor,pos=0.5,right,left] {$$};

\draw[-] (5-2-4) -- (6-3-6) node[\posetedgecolor,pos=0.5,right,left] {$$};
\draw[-] (3-3-4) -- (4-4-6) node[\posetedgecolor,pos=0.5,right,left] {$$};

\draw[thick,dotted] (-.3,-8.2)..controls (-.5,-4.5) and  (.5, -4.5) ..(.3,-8.2);

\draw[thick,dotted] (-1,-8.2)..controls (-.5,-3.25) and  (.5, -3.25) ..(.75,-8.2);

\draw[thick,dotted] (-1.5,-8.1)..controls (-.75,-3) and  (0, -3) ..(1, -4.25);

\draw[thick,dotted] (1,-4.25)..controls (1.25,-4.5) and  (1.5, -4.75) ..(1.75,-8.1);

\node(a) at (.25,-8.4) {$\Gamma_3$};
\node at (.8,-8.4) {$\Gamma_2$};
\node at (1.8,-8.4) {$\Gamma_1$};
\end{tikzpicture}
\caption{The trie of subwords of $\color{blue}10010111$ (left) and the antichain trie of the corresponding poset (right)}
\label{fig:subwordtrie_10010111} 
\label{fig:antichaintrie_10010111} 
\begin{tikzpicture}[xscale=1.8,yscale=1,>=latex]
\def\posetedgecolor{blue} 
\node(empty) at (0,0) {$\emptyset$};
\node(0-1) at (0,-1) {$\{1\}$};
\node(0-2) at (0,-2) {$\{2\}$};
\node(0-3) at (0,-3) {$\{3\}$};
\node(0-4) at (0,-4) {$\{4\}$};
\node(0-5) at (0,-5) {$\{5\}$};
\node(0-6) at (0,-6) {$\{6\}$};
\node(0-7) at (0,-7) {$\{7\}$};
\node(0-8) at (0,-8) {$\{8\}$};

\node(1-5-6) at (1,-5) {$\{4,6\}$};
\node(1-6-7) at (1,-6) {$\{4,7\}$};
\node(1-7-8) at (1,-7) {$\{4,8\}$};

\node(2-4-5) at (2,-4) {$\{3,5\}$};
\node(2-5-6) at (2,-5) {$\{3,6\}$};
\node(2-6-7) at (2,-6) {$\{3,7\}$};
\node(2-7-8) at (2,-7) {$\{3,8\}$};

\node(3-3-4) at (3,-3) {$\{2,4\}$};
\node(3-4-5) at (3,-4) {$\{2,5\}$};
\node(3-5-6) at (3,-5) {$\{2,6\}$};
\node(3-6-7) at (3,-6) {$\{2,7\}$};
\node(3-7-8) at (3,-7) {$\{2,8\}$};

\node(4-4-6) at (4,-4) {$\{2,4,6\}$};
\node(4-5-7) at (4,-5) {$\{2,4,7\}$};
\node(4-6-8) at (4,-6) {$\{2,4,8\}$};

\node(5-2-4) at (5,-2) {$\{1,4\}$};
\node(5-3-5) at (5,-3) {$\{1,5\}$};
\node(5-4-6) at (5,-4) {$\{1,6\}$};
\node(5-5-7) at (5,-5) {$\{1,7\}$};
\node(5-6-8) at (5,-6) {$\{1,8\}$};

\node(6-3-6) at (6,-3) {$\{1,4,6\}$};
\node(6-4-7) at (6,-4) {$\{1,4,7\}$};
\node(6-5-8) at (6,-5) {$\{1,4,8\}$};

\draw[-] (empty) 
-- (0-1) node[\posetedgecolor,pos=0.5,right,left] {}
 -- (0-2) node[\posetedgecolor,pos=0.2,below] {} 
 -- (0-3) node[\posetedgecolor,pos=0.5,above] {}
 -- (0-4) node[\posetedgecolor,pos=0.5,left] {}
 -- (0-5) node[\posetedgecolor,pos=0.5,right=1pt] {}
 -- (0-6) node[\posetedgecolor,pos=0.5,left=1mm] {}
 -- (0-7) node[\posetedgecolor,pos=0.5,left=1mm] {}
 -- (0-8) node[\posetedgecolor,pos=0.5,left=1mm] {};

\draw[-] 
-- (1-5-6) node[\posetedgecolor,pos=0.5,right,left] {}
 -- (1-6-7) node[\posetedgecolor,pos=0.2,below] {} 
 -- (1-7-8) node[\posetedgecolor,pos=0.5,above] {};
 
 \draw[-] 
-- (2-4-5) node[\posetedgecolor,pos=0.5,right,left] {}
 -- (2-5-6) node[\posetedgecolor,pos=0.2,below] {} 
 -- (2-6-7) node[\posetedgecolor,pos=0.5,above] {}
 --(2-7-8) node[\posetedgecolor,pos=0.5,above] {};
 
  \draw[-] 
-- (3-3-4) node[\posetedgecolor,pos=0.5,right,left] {}
 -- (3-4-5) node[\posetedgecolor,pos=0.2,below] {} 
 -- (3-5-6) node[\posetedgecolor,pos=0.5,above] {}
 --(3-6-7) node[\posetedgecolor,pos=0.5,above] {}
 --(3-7-8) node[\posetedgecolor,pos=0.5,above]{};

 \draw[-] 
-- (4-4-6) node[\posetedgecolor,pos=0.5,right,left] {}
 -- (4-5-7) node[\posetedgecolor,pos=0.2,below] {} 
 -- (4-6-8) node[\posetedgecolor,pos=0.5,above] {};

 \draw[-] 
-- (5-2-4) node[\posetedgecolor,pos=0.5,right,left] {}
 -- (5-3-5) node[\posetedgecolor,pos=0.2,below] {} 
 -- (5-4-6) node[\posetedgecolor,pos=0.5,above] {}
 --(5-5-7) node[\posetedgecolor,pos=0.5,above] {}
 --(5-6-8) node[\posetedgecolor,pos=0.5,above] {};
 
  \draw[-] 
-- (6-3-6) node[\posetedgecolor,pos=0.5,right,left] {}
 -- (6-4-7) node[\posetedgecolor,pos=0.2,below] {} 
 -- (6-5-8) node[\posetedgecolor,pos=0.5,above] {};

\draw[-] (0-1) -- (5-2-4) node[\posetedgecolor,pos=0.5,right,left] {};
\draw[-] (0-2) -- (3-3-4) node[\posetedgecolor,pos=0.5,right,left] {};
\draw[-] (0-3) -- (2-4-5) node[\posetedgecolor,pos=0.5,right,left] {};
\draw[-] (0-4) -- (1-5-6) node[\posetedgecolor,pos=0.5,right,left] {};

\draw[-] (5-2-4) -- (6-3-6) node[\posetedgecolor,pos=0.5,right,left] {};
\draw[-] (3-3-4) -- (4-4-6) node[\posetedgecolor,pos=0.5,right,left] {};
\end{tikzpicture}
\caption{Antichains of the poset of Fig.~\protect\ref{hasse_10010111} associated to nodes of the antichain trie of  Fig.~\protect\ref{fig:antichaintrie_10010111} (right)}\label{fig:antichainsof10010111}
\end{center}
\end{figure}

\section{Antichain trie}
\label{sec:antichain_trie}
Given a piecewise-linear poset $P$, we construct an antichain analog, called the \defn{antichain trie} of $P$, of the~\cite{LRS17} trie of subwords. The nodes of the antichain trie are in one-to-one correspondence with the antichains of $P$. 

\begin{definition}[Antichain trie]
Let $P$ be a piecewise-linear poset with elements $P_1$, $P_2$, $\dots$, $P_d$, arranged from left to right in the Hasse diagram of $P$ as in Fig.~\ref{fig:hasse101110left} (left).
We construct the antichain trie of $P$ as follows.

Let N+2 be the number of minimal/maximal elements of $P$. 
Reading in order from left to right on the Hasse diagram, denote the  minimal and maximal elements $E_n, \, n=0, 1, 2, \dots,  N, N+1$. 

Begin with a vertical linear tree $\Gamma_P$ with $d+1$ nodes labeled by $0, 1 ,\dots, d$. Let $\Gamma_P$ be rooted at $\epsilon = 0$ and let the node with label $i$ be the left child of the node labeled $i-1$ for all $i = 1, 2, \dots, d$.

If $N$=0, then the antichain trie is $\Gamma_P$. 
Otherwise, starting from the bottom of $\Gamma_P$, we define a tree $\Gamma_n$ for every $n = N, N-1, \dots, 1$ .

First, let $\Gamma_N$ be the (linear) subtree of $\Gamma_P$ consisting of $P_i, P_{i+1}, \dots , P_d= E_{N+1}$ where $P_i$ is the element immediately to the right of $E_N$ in the Hasse diagram of $P$. We then attach a copy of $\Gamma_{N}$ to each node (on the vertical linear tree $\Gamma_P$) whose label corresponds to the elements of $P$ between $E_{N}$ and $E_{N-1} $, excluding $E_N$ but including $E_{N-1}$. 
Let the root of each copy of $\Gamma_N$ be the right child of the node of $\Gamma_P$ that this root is attached to. 
Note that if there is only one element to the right of $E_{N}$ (that is, if $E_N=P_{d-1}$, as in Fig.~\ref{fig:hasse101110left} (left) and Fig.~ \ref{fig:hassee11001110}), then we attach a single node labeled $P_d$. 
This process results in a (non-linear tree) $\Gamma_P'$ that is larger than $\Gamma_P$. 

Next, let $\Gamma_{N-1}$ be the maximal subtree of this larger tree $\Gamma_P'$ such that its root is labeled by the element of $P$ directly to the right of $E_{N-1}$ (in the Hasse diagram of $P$). We then attach a copy of $\Gamma_{N-1}$ to each node of $\Gamma_P$ whose label corresponds to the elements between $E_{N-1}$ and $E_{N-2}$, excluding $E_{N-1}$ but including $E_{N-2}$.

Let $\Gamma_{N-2}$ be the subtree of this larger tree such that its root corresponds to the element of $P$ directly to the right of $E_{N-2}$ (in the Hasse diagram of $P$) and $\Gamma_{N-2}$ contains all descendants of this root.

Continue as such until we have attached copies of $\Gamma_1$ to the nodes on $\Gamma_P$ whose labels correspond to the elements between $E_1$ and the element $E_0 = P_1$, excluding $E_1$ but including~$P_1$.

\begin{remark}Note that we always add copies of $\Gamma_1$, in contrast to the construction of the trie of subwords in the previous section.
\end{remark}

We now describe a bijective map from the nodes of the antichain trie of $P$ to the antichains of $P$.
To each node $v$, we associate the path $\epsilon, v_1, \dots,v_{\ell}$ along the vertices from $\epsilon$ to $v$. Let $p(v)$ be the sequence of labels $L(\epsilon)=0, L(v_1), \dots,L(v_\ell)$ of these vertices. 
Note that $p(v)$ is an ordered subsequence of $(0,1, \dots,d)$ which is uniquely determined by $v$. 
Let $A(\epsilon)=\emptyset$, and for the rest of the nodes $v$, let  
\begin{align} 
\label{eq:path_to_an_antichain}
A(v)&=\Big\{ j \in \{ 1,2,\dots, d\} ~ |\\
\nonumber
 & \text{$j$ is the largest number in a block of consecutive integers in $p(v)$} \Big\}.\end{align}
\end{definition}

\begin{proposition}
\label{prop:antichain_trie}
Let $P$ be a piecewise-linear poset. 
Then for each node $v$ of the antichain trie of $P$, the set $A(v)$ is a distinct antichain of $P$.
\end{proposition}

\begin{example} 
In Fig.~\ref{fig:trie_antichainsof101110} (right), we have labeled every node $v$ with $A(v)$. For example, if $v$ is the node obtained by the path $p(v) = (0, 1, 2, 3, 6 )$ of Fig.~\ref{fig:ex:antichaintrie_101110} (center), then $A(v) = \{3,6\}$. For the node $v$ obtained by walking along $p(v) = (0, 1, 3, 4, 6)$, we have the antichain $A(v) = \{1, 4, 6\}$. 
\end{example}

\begin{example} \label{ex:antichainsof10010111}
Let $w= 10010111$. 
Fig.~\ref{hasse_10010111} shows the Hasse diagram of the eight-element poset $P$ corresponding to $w$. Fig.~\ref{fig:subwordtrie_10010111} (left) depicts the trie of subwords for $w$. Write  $w =\underbrace{1^1}_{u_1}\underbrace{0^2}_{u_2}\underbrace{1^1}_{u_3} \underbrace{0^1}_{u_4}\underbrace{1^3}_{u_5} $, so $M = 5$. The root of $T_4$ is $u_1u_2u_3u_41$ (the diamond node), the root of $T_3$ is $u_1u_2u_30$ (the square node), and the root of $T_2$ is $u_1u_21$ (the 
star node).

The antichain trie of the poset $P$ corresponding to $w$ is given in Fig.~\ref{fig:antichaintrie_10010111} (right). 
Fig.~\ref{fig:antichainsof10010111} shows the $32$ antichains of $P$ 
which are assigned to the $32$ nodes of the antichain trie by~ (\ref{eq:path_to_an_antichain}).  
\end{example}

\begin{example}
Let $P$ be the eight-element poset whose Hasse diagram is illustrated in Fig.~\ref{fig:hassee11001110}.  
The antichain trie for $P$ is given in Fig.~\ref{fig:trieofantichains_11001110}. 
Note that $P$ has 5 minimal/ maximal elements so $N=3$. Our vertical linear tree $\Gamma_P$ has 9 nodes labeled  $0, 1, \dots, 8$. In this example,  $ \Gamma_3$ is the linear subtree consisting solely of $P_8 =8$. We then attach copies of $\Gamma_3$ to the nodes of $\Gamma_P$ labeled by $E_2 = P_4, P_5$ and $P_6$ which are $4, 5, 6$, respectively. Next, we construct $\Gamma_2$ which is rooted at the node on $\Gamma_P$ labeled by $P_5 =5$. We attach copies of $\Gamma_2$ to the nodes of $\Gamma_P$ labeled by $E_1=P_2$ and $P_3$. Lastly, $\Gamma_1$ has its root at the node of $\Gamma_P$ labeled $P_3=3$.  In this example, a single copy of $\Gamma_1$ is attached to the element on $\Gamma_P$ labeled by $P_1=1$. The trie of subwords for the corresponding word $w = 11001110$ is given in \cite[Figs.~3-4]{LRS17}.

\begin{figure}[hbpt]
\begin{tikzpicture}[xscale=1, yscale=1]
\def\posetedgecolor{blue}
\node(1) at (0,0) {{$1$}}; 
\node(2) at (1,1) {{$2$}};
\node(3) at (2,0) {{$3$}};
\node(4) at (3,-1) {{$4$}};
\node(5) at (4,0) {{$5$}};
\node(6) at (5,1) {{$6$}};
\node(7) at (6,2) {{$7$}};
\node(8) at (7,1) {{$8$}};

\node[\posetedgecolor] at (-0.9,0) {$\underline{1}$};
\draw[-, line width=1pt] (2) -- (1) node[\posetedgecolor,pos=0.7,above] {\small $1$}; 
\draw[-, line width=1pt]  (2) -- (3) node[\posetedgecolor,pos=0.2,below] {\small $0$};  
\draw[-, line width=1pt] (3) -- (4) node[\posetedgecolor,pos=0.2,below] {\small $0$};  
\draw[-, line width=1pt] (4) -- (5) node[\posetedgecolor,pos=0.2,above] {\small $1$}; 
\draw[-,line width=1pt] (6) -- (5) node[\posetedgecolor,pos=0.7,above] {\small $1$}; 
\draw[-,line width=1pt] (6) -- (7) node[\posetedgecolor,pos=0.7,above] {\small $1$};
\draw[-,line width=1pt] (7) -- (8) node[\posetedgecolor,pos=0.2,below] {\small $0$};
\end{tikzpicture}

\caption{The Hasse diagram of the poset associated to the word \textcolor{blue}{$11001110$}}
\label{fig:hassee11001110}

\begin{center}
\begin{tikzpicture}[xscale=1.2,yscale=1,>=latex]
\def\posetedgecolor{blue}

\node(empty) at (0,0) {$0$};
\node(0-1) at (0,-1) {$1$};
\node(0-2) at (0,-2) {$2$};
\node(0-3) at (0,-3) {$3$};
\node(0-4) at (0,-4) {$4$};
\node(0-5) at (0,-5) {$5$};
\node(0-6) at (0,-6) {$6$};
\node(0-7) at (0,-7) {$7$};
\node(0-8) at (0,-8) {$8$};

\node(1-7-2) at (1,-7) {$8$};
\node(1-6-2) at (1,-6) {$8$};
\node(1-5-2) at (1,-5) {$8$};

\node(2-2-3) at (2,-4) {$5$};
\node(2-3-3) at (2,-5) {$6$};
\node(2-4-3) at (2,-6) {$7$};
\node(2-5-3) at (2,-7) {$8$};

\node(3-6-4) at (3,-6) {$8$};
\node(3-5-4) at (3,-5) {$8$};

\node(4-6-5) at (4,-6) {$8$};
\node(4-5-5) at (4,-5) {$7$};
\node(4-4-5) at (4,-4) {$6$};
\node(4-3-5) at (4,-3) {$5$};

\node(5-5-6) at (5,-5) {$8$};
\node(5-4-6) at (5,-4) {$8$};

\node(6-7-7) at (6,-7) {$8$};
\node(6-6-7) at (6,-6) {$7$};
\node(6-5-7) at (6,-5) {$6$};
\node(6-4-7) at (6,-4) {$5$};
\node(6-3-7) at (6,-3) {$4$};
\node(6-2-7) at (6,-2) {$3$};

\node(7-6-8) at (7,-6) {$8$};
\node(7-5-8) at (7,-5) {$8$};
\node(7-4-8) at (7,-4) {$8$};

\node(8-6-9) at (8,-6) {$8$};
\node(8-5-9) at (8,-5) {$7$};
\node(8-4-9) at (8,-4) {$6$};
\node(8-3-9) at (8,-3) {$5$};

\node(9-5-10) at (9,-5) {$8$};
\node(9-4-10) at (9,-4) {$8$};

\draw[-] (2-2-3) 
-- (2-3-3) node[\posetedgecolor,pos=0.5,right,left] {}
 -- (2-4-3) node[\posetedgecolor,pos=0.2,below] {} 
 -- (2-5-3) node[\posetedgecolor,pos=0.5,above] {};
 
 \draw[-] (1-5-2) --(0-4)node[\posetedgecolor,pos=0.5,right,left] {};
 \draw[-] (1-6-2) --(0-5)node[\posetedgecolor,pos=0.5,right,left] {};
  \draw[-] (1-7-2) --(0-6)node[\posetedgecolor,pos=0.5,right,left] {};

\draw[-] (empty) 
-- (0-1) node[\posetedgecolor,pos=0.5,right,left] {}
 -- (0-2) node[\posetedgecolor,pos=0.2,below] {} 
 -- (0-3) node[\posetedgecolor,pos=0.5,above] {}
-- (0-4) node[\posetedgecolor,pos=0.5,left] {}
-- (0-5) node[\posetedgecolor,pos=0.5,right,left] {}
-- (0-6) node[\posetedgecolor,pos=0.5,right,left] {}
-- (0-7) node[\posetedgecolor,pos=0.5,right,left] {}
 -- (0-8) node[\posetedgecolor,pos=0.5,left=1mm] {};
 
 \draw[-]( 0-3) --(2-2-3)node[\posetedgecolor,pos=0.5,right,left] {};
 
 \draw[-] (2-2-3) --(3-5-4)node[\posetedgecolor,pos=0.5,right,left] {};
 
 \draw[-] (2-3-3) -- (3-6-4)node[\posetedgecolor,pos=0.5,right,left] {};
 
 \draw[-] (0-2) -- (4-3-5) node[\posetedgecolor,pos=0.5,right,left] {};
 
\draw[-] (4-3-5) 

 -- (4-4-5) node[\posetedgecolor,pos=0.2,below] {} 
 -- (4-5-5) node[\posetedgecolor,pos=0.5,above] {}
 --(4-6-5) node[\posetedgecolor,pos=0.5,above] {};
 
 \draw[-] (0-1) -- (6-2-7)node[\posetedgecolor,pos=0.5,right,left] {};
 
 \draw[-] (6-2-7) 
-- (6-3-7) node[\posetedgecolor,pos=0.5,right,left] {}
 -- (6-4-7) node[\posetedgecolor,pos=0.2,below] {} 
 -- (6-5-7) node[\posetedgecolor,pos=0.5,above] {}
  -- (6-6-7) node[\posetedgecolor,pos=0.5,above] {}
-- (6-7-7) node[\posetedgecolor,pos=0.5,left] {};

 \draw[-] (8-3-9) 
-- (8-4-9) node[\posetedgecolor,pos=0.5,right,left] {}
 -- (8-5-9) node[\posetedgecolor,pos=0.2,below] {} 
 -- (8-6-9) node[\posetedgecolor,pos=0.5,above] {};

 \draw[-] (4-3-5) -- (5-4-6) node[\posetedgecolor,pos=0.5,right,left] {};
 \draw[-] (4-4-5) -- (5-5-6) node[\posetedgecolor,pos=0.5,right,left] {};
  \draw[-] (6-3-7) -- (7-4-8) node[\posetedgecolor,pos=0.5,right,left] {};
    \draw[-] (6-4-7) -- (7-5-8) node[\posetedgecolor,pos=0.5,right,left] {};
    \draw[-] (6-5-7) -- (7-6-8) node[\posetedgecolor,pos=0.5,right,left] {};  
   \draw[-] (6-2-7) -- (8-3-9) node[\posetedgecolor,pos=0.5,right,left] {};
     \draw[-] (8-3-9) -- (9-4-10) node[\posetedgecolor,pos=0.5,right,left] {};
       \draw[-] (8-4-9) -- (9-5-10) node[\posetedgecolor,pos=0.5,right,left] {};

\draw[thick, dotted] (0,-8) circle (5mm);      
\draw[thick,dotted] (-1,-8.5)..controls (-.5,-2.25) and  (.5, -5) ..(1,-5.5);

\draw[thick, dotted]
(1,-5.5) .. controls (1.25, -5.5) and (1.5, -5.5).. (2, -8.5);

\draw[thick, dotted]
(-1.5, -8.5) .. controls (-.5, .5) and (.25,-3).. (3, -3.5); 

\draw[thick, dotted] (3, -3.5).. controls (3.25, -3.75) and (3.5, -4)..(4, -8.5); 

\node(a) at (0,-8.8) {$\Gamma_3$};

\node(b) at (2,-8.8) {$\Gamma_2$};

\node(c) at (4,-8.8) {$\Gamma_1$};

\end{tikzpicture}
\caption{ The antichain trie of the poset of Fig.~\protect\ref{fig:hassee11001110} 
}
\label{fig:trieofantichains_11001110}
\end{center}
\end{figure}
\end{example}

\begin{remark}
Given an antichain trie, 
let a \defn{vertical branch} be a linear subtree with nodes $v_1,\dots,v_{r+1}$ ($r\geq 1$) so that $v_1$ is either the root $\epsilon$ or is the right child of its parent (in particular, it is not a left child), $v_{i+1}$ is the left child of $v_i$ for $i=1,\dots,r$,
and $v_{r+1}$ is not a parent. 
For example, in Fig.~\ref{fig:ex:antichaintrie_101110} (center), the original left-most vertical tree $\Gamma_P$ and the subtree labeled $\{3,4,5,6\}$ are 
the only vertical branches 
while Fig.~\ref{fig:antichaintrie_10010111} (right) has $7$ vertical branches.  
Note that, by construction of the antichain trie, 
a \defn{vertical move} (downward) from a node $v$ to its left child $v'$ 
removes the label $L(v)$ from $A(v)$ and replaces it with the label $L(v')=L(v)+1$.

Similarly, let a \defn{horizontal branch} be a linear subtree with nodes $v_1,\dots,v_{r+1}$ ($r\geq 1$) so that $v_1$ is the left child of some node (in particular, $v_1$ cannot be a right child), $v_{i+1}$ is the right child of $v_i$ for all $i=1,\dots, r$, 
and $v_{r+1}$ does not have a right child. For example, in Fig.~\ref{fig:ex:antichaintrie_101110} (center), the subtrees labeled $\{1,3,6\}$, $\{2,6\}$, $\{3,6\}$, and $\{4,6\}$ are the horizontal branches, with $2$ different horizontal branches both labeled by $\{4,6\}$. 
The antichain trie of Fig.~\ref{fig:antichaintrie_10010111} (right) has exactly $4$ horizontal branches, labeled by $\{1,4, 6 \}$, 
$\{ 2,4,6\}$, $\{ 3, 5 \}$, and $\{ 4, 6\}$.  
A \defn{horizontal move} (to the right) from a node $v$ to its right child $v'$ adds to the antichain $A(v)$ the (positive integer) label $L(v')$ of $v'$ which is greater than $L(v)+1$. 
\end{remark}

\section{Bijection between antichains and subwords}\label{sec:bijection_antichains_subwords}
\label{sec:bijection_between_antichains_and_subwords}
Let $w=w_1 \dots w_d$ be a binary word of length $d$. Let $P=\{ P_1=1,\dots, P_d=d\}$ be the corresponding piecewise-linear poset whose Hasse diagram $H$ has edges labeled by $w_2,\dots,w_d$.
We now define a map $f$ from the antichains in $P$ to the subwords of $w$. 
\begin{definition}[]\label{def:injective_map}
Let $f(\emptyset)$ be the empty subword. 
If $A = \{A_1, A_2 \dots, A_r\}$ is a nonempty antichain in $P$,   
let $f(A)$ be the subword of $w$ which is constructed as follows.  The first letter is 1. The next letters are the (possibly empty) sequence of edge labels of $H$ between $P_1$ and $A_1$. If $A$ contains one element, we are done. If $A$ contains more than one element, jump to the first minimal or maximal element $M_1$ appearing after $A_1$. Record the labels of edges between $M_1$ and $A_2$. Next, jump to the first minimal or maximal element $M_2$ appearing after $A_2$. Record the labels of edges between $M_2$ and $A_3$. Continue as such until we finish  recording the edge labels between $M_{r-1}$ and $A_r$.
\end{definition}

\def\aa{1}
\def\bb{2}
\def\cc{3}
\def\dd{4}
\def\ee{5}
\def\ff{6}
\def\gg{7}
\def\hh{8}
\def\ii{9}
\def\jj{10}
\begin{example}
\label{example:antichain_to_subword}
\def\AntichainToSubwordScaleX{0.75}
\def\AntichainToSubwordScaleY{1}
In Fig.~\ref{fig:example:antichain_to_subword}, 
the antichains $A$=$\{ A_1$=$\circled{\dd}$, $A_2$=$\circled{\jj}\}$ 
and $A$=$\{ A_1$=$\circled{\aa}$, $A_2$=$\circled{\cc}$, $A_3$=$\circled{\gg}$, $A_4$=$\circled{\ii}\}$ 
are mapped to the subwords 
s=\textcolor{blue}{\underline{1}\boxed{011}\boxed{01100}} and 
s=\textcolor{blue}{\underline{1}\framebox(7,12){}\framebox(10,12){1}\framebox(20,12){01}\framebox(10,12){0}},
respectively,
 of
$w=1011101100$. 

\begin{figure}[hbpt]
\begin{center}
\begin{tikzpicture}[xscale=\AntichainToSubwordScaleX,
yscale = \AntichainToSubwordScaleY, 
>=latex,
font = \small
]
\def\posetedgecolor{blue}
\node(1) at (0,0) {{$\aa$}}; 
\node(2) at (1,-1) {{$\bb$}}; 
\node(3) at (2,0) {{$\cc$}}; 
\node(4) at (3,1) {\circled{$\dd$}}; 
\node(5) at (4,2) {{$\ee=M_1$}}; 
\node(6) at (5,1) {{$\ff$}}; 
\node(7) at (6,2) {{$\gg$}}; 
\node(8) at (7,3) {{$\hh$}}; 
\node(9) at (8,2) {{$\ii$}}; 
\node(10) at (9,1) {\circled{$\jj$}}; 

\node[\posetedgecolor] at (-0.7,0) {$\underline{1}$};
\draw[-] (2) -- (1) node[\posetedgecolor,pos=0.5] {\small $0$}; 
\draw[-]  (2) -- (3) node[\posetedgecolor,pos=0.5] {\small $1$}; 
\draw[-] (3) -- (4) node[\posetedgecolor,pos=0.5] {\small $1$};  
\draw[-] (4) -- (5); 
\draw[-] (6) -- (5) node[\posetedgecolor,pos=0.5] {\small $0$}; 
\draw[-] (6) -- (7) node[\posetedgecolor,pos=0.5] {\small $1$};  
\draw[-] (7) -- (8) node[\posetedgecolor,pos=0.5] {\small $1$};  
\draw[-] (9) -- (8) node[\posetedgecolor,pos=0.5] {\small $0$};  
\draw[-] (10) -- (9) node[\posetedgecolor,pos=0.5] {\small $0$}; 
\end{tikzpicture}
\hspace{0mm}
\begin{tikzpicture}
[xscale=\AntichainToSubwordScaleX, 
yscale = \AntichainToSubwordScaleY, 
>=latex, 
]
\def\posetedgecolor{blue}
\node(1) at (0,0) {\circled{$\aa$}}; 
\node(2) at (1,-1) {{$\bb=M_1$}}; 
\node(3) at (2,0) {\circled{$\cc$}}; 
\node(4) at (3,1) {{$\dd$}}; 
\node(5) at (4,2) {{$\ee=M_2$}}; 
\node(6) at (5,1) {{$\ff$}}; 
\node(7) at (6,2) {\circled{$\gg$}}; 
\node(8) at (7,3) {{$\hh=M_3$}}; 
\node(9) at (8,2) {\circled{$\ii$}}; 
\node(10) at (9,1) {{$\jj$}}; 

\node[\posetedgecolor] at (-0.9,0) {$\underline{1}$};
\draw[-] (2) -- (1) ; 
\draw[-]  (2) -- (3) node[\posetedgecolor,pos=0.5] {\small $1$}; 
\draw[-] (3) -- (4) ; 
\draw[-] (4) -- (5) ;
\draw[-] (6) -- (5) node[\posetedgecolor,pos=0.5] {\small $0$}; 
\draw[-] (6) -- (7) node[\posetedgecolor,pos=0.5] {\small $1$};  
\draw[-] (7) -- (8) ; 
\draw[-] (9) -- (8) node[\posetedgecolor,pos=0.5] {\small $0$};  
\draw[-] (10) -- (9); 
\end{tikzpicture}
\caption{
The antichains mapped to 
s=\textcolor{blue}{\underline{1}\boxed{011}\boxed{01100}} 
(left)
and 
s=\textcolor{blue}{\underline{1}\framebox(7,12){}\framebox(10,12){1}\framebox(20,12){01}\framebox(10,12){0}} 
(right) of 
$w=1011101100$
}
\label{fig:example:antichain_to_subword}
\end{center}
\end{figure}
\end{example}

\begin{theorem}
\label{thm:antichain_to_subword}
The map $f$ given in Definition~\ref{def:injective_map} is a bijection from the antichains in $P$ to the subwords of $w$. 
\end{theorem}
\begin{proof}
To show that $f$ is surjective, let $s$ be a subword of $w=w_1 w_2 \dots w_d$. If $s$ is nonempty, write $s=w_{i_1}w_{i_2} \dots w_{i_\ell}$ in such a way that each index $i_k$ is as small as possible (see Example~\ref{ex:inverse_map_11010}). Note that $w_{i_1}=w_1=1$ per our definition of subwords. 
Partition $\{ w_{i_1}, w_{i_2}, \dots, w_{i_\ell} \}$ into a set $\Sigma=\Sigma_s$ of (at least one)  maximal blocks of subsequences of $w$ such that each subsequence is a consecutive subsequence.  

Let \[A=A_\Sigma=\{ j \in P | (w_i, w_{i+1} \dots, w_j) \in \Sigma \}.\]
In other words, $(w_1)\in \Sigma$ if and only if $1 \in A$; if $2\leq n \leq d$, then $n\in A$ if and only if the node $n$ in the Hasse diagram $H$ is immediately to the right of a block of edges in $\Sigma$.

We claim that $A$ is an antichain.
If $\Sigma$ only contains one block, then $A$ consists of one element, and hence $A$ is an antichain in $P$.
Otherwise, let $(w_i,\dots,w_j)$, where $3\leq i\leq d$, be a second block in $\Sigma$. 
If $w_i=0$, then $w_{i-1}=1$ since the indices $i_k$'s for the $w_{i_k}$'s were chosen to be as small as possible. Likewise, if $w_i=1$, then $w_{i-1}=0$. This means that the node $i-1$ (which is not in $A$) between edges $w_{i-1}$ and $w_{i}$ is either a minimal or maximal element of $P$. Hence no node to the left of $i-1$ is related to the node $j$. Similarly, if there is another block $(w_{i'},\dots,w_{j'})$ of $\Sigma$ which appears after  $(w_i,\dots,w_j)$, the node $j$ is not related to the node $j'$. This shows that $j$ is not related to any other element in $A$.

To show that the map is injective, assume $f(A) = f(A')$. Then $f(A) = s = f(A')$ for some subword $s=w_{i_1} \dots w_{i_\ell}$. Let $\Sigma_s$ be the set of maximal blocks of $w_{i_1}, \dots, w_{i_\ell}$ as defined on the first paragraph of this proof. But both $A$ and $A'$ are defined by the same set $\Sigma_s$ of maximal blocks, so $A = A'$.
\end{proof}

\def\myyscale{0.735}
\def\myxscale{1}

\begin{example}\label{ex:inverse_map_11010}
Consider the word $w=w_1 \dots w_{10} = 1011101100$. 
Identify $w_2,\dots,w_{10}$ with the edges of the Hasse diagram $H$ of $P$, see Fig.~\ref{fig:poset1011101100}. 
We write the subword s=\textcolor{blue}{$11010$} as $s=w_1 w_3 w_6 w_7 w_9$ so that the index of each letter $w_{i_k}$ is as small as possible.  Partitioning $\{ w_1, w_3, w_6, w_7, w_9 \}$ into maximal blocks of consecutive subsequences of $w$ gives four blocks $(w_1)$, $(w_3)$, $(w_6, w_7)$, and $(w_9)$. We build an antichain as follows. Since $(w_1)$ is a block, we take the left-most node of $H$, node $1$.
We take the nodes of $H$ to the right of the other three blocks, nodes $3$, $7$, and $9$. Therefore, the subword 11010 corresponds to the antichain $\{
\circled{1},\circled{3},\circled{7},\circled{9}\}$.
\end{example}

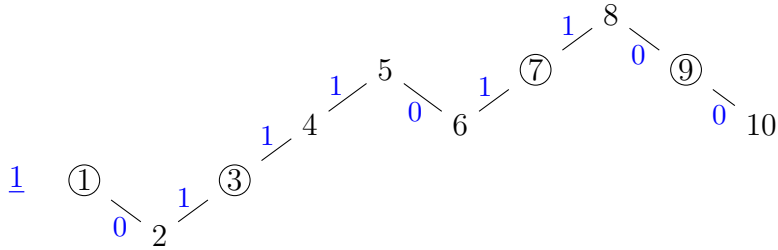
\begin{figure}[h]
\begin{center}
\begin{tikzpicture}[xscale=\myxscale,yscale=\myyscale,>=latex]
\def\posetedgecolor{blue}
\node(1) at (0,0) {\circled{$1$}}; 
\node(2) at (1,-1) {$2$}; 
\node(3) at (2,0) {\circled{$3$}}; 
\node(4) at (3,1) {$4$}; 
\node(5) at (4,2) {$5$}; 
\node(6) at (5,1) {$6$}; 
\node(7) at (6,2) {\circled{$7$}}; 
\node(8) at (7,3) {$8$};
\node(9) at (8,2) {\circled{$9$}};
\node(10) at (9,1)  {$10$}; 

\node[\posetedgecolor] at (-0.9,0) {$\underline{1}$};
\draw[-] (2) -- (1) node[\posetedgecolor,pos=0.7,below] {\small $0$}; 
\draw[-]  (2) -- (3) node[\posetedgecolor,pos=0.2,above] {\small $1$}; ; 
\draw[-] (3) -- (4) node[\posetedgecolor,pos=0.2,above] {\small $1$}; ; 
\draw[-] (4) -- (5) node[\posetedgecolor,pos=0.2,above] {\small $1$}; ;
\draw[-] (6) -- (5) node[\posetedgecolor,pos=0.7,below] {\small $0$}; ; 
\draw[-] (6) -- (7) node[\posetedgecolor,pos=0.2,above] {\small $1$}; ; 
\draw[-] (7) -- (8) node[\posetedgecolor,pos=0.2,above] {\small $1$}; ; 
\draw[-] (9) -- (8) node[\posetedgecolor,pos=0.7,below] {\small $0$}; ; 
\draw[-] (10) -- (9)node[\posetedgecolor,pos=0.7,below] {\small $0$}; ;
\end{tikzpicture}
\caption{
Hasse diagram representing the word $\color{blue}1011101100$; the antichain corresponding to the subword 
\textcolor{blue}{$\underline{1}\boxed{1}\boxed{01}\boxed{0}$} 
is $\{ 1,3,7,9\}$
}
\label{fig:poset1011101100}
\end{center}
\end{figure}

\section{Subwords to snake graph matchings}
\label{sec:subwords_to_snake_graph_matchings}

\subsection{Background}
We review the theory of snake graphs developed in~\cite{Pro05, MSW11, MSW13, CS13}.

\begin{definition}
A \defn{snake graph} is a nonempty connected sequence of square tiles~
\begin{tikzpicture}[scale=.25]
\draw (0,0) -- (1,0) -- (1,1) -- (0,1) -- cycle;
\end{tikzpicture}.
To build a snake graph $G$ with $d$ tiles, start with one tile, then glue a new tile to the north or the east of the previous tile. We refer to the southwest-most tile of $G$ as the first tile $G_1$ and the northeast-most tile as the last tile $G_d$. Fig.~\ref{fig:minimalmatching} (left) illustrates a snake graph with $10$ tiles.
\end{definition}

\begin{definition}
A \defn{matching} of a graph $G$ is a subset of non-adjacent edges of $G$.
A \defn{perfect matching} of $G$ is a matching where every vertex of $G$ is adjacent to exactly one edge of the matching, see Fig.~\ref{fig:snake_graph_matching}.
Define the \defn{minimal matching} $pm_{\text{min}}$ to be the unique perfect matching of $G$ which contains the south edge of the first tile $G_1$ 
 and only boundary edges, see Fig.~\ref{fig:minimalmatching} ~(center).

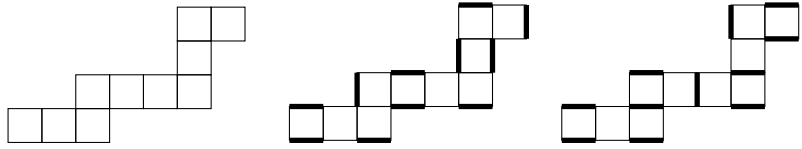
\begin{figure}[h]
\begin{center}
\def\myscale{0.45}
\newcommand\macropositivetile{
     \def\south{}
     \def\east{}
     \def\north{}
     \def\west{} 
     \draw (\x + 0, \y) --  (\x+1,\y) node [pos=0.5,above] {\south} -- (\x+1,\y+1) node [pos=0.5,left] {\east} -- (\x+0,\y+1) node [pos=0.5,above] {\north} --  (\x + 0, \y)  node [pos=0.5,left] {\west};
}

\newcommand\macronegativetile{
     \def\south{}
     \def\east{}
     \def\north{}
     \def\west{} 
     \draw (\x + 0, \y) --  (\x+1,\y) node [pos=0.5,above] {\south} -- (\x+1,\y+1) node [pos=0.5,left] {\east} -- (\x+0,\y+1) node [pos=0.5,above] {\north} --  (\x + 0, \y)  node [pos=0.5,left] {\west};
}
 \begin{tikzpicture}[scale=\myscale]
 \foreach \x in {0,1,...,2} 
 {
     \draw (\x + 0, 0) -- (\x+1,0) -- (\x+1,1)-- (\x+0,1) -- cycle;
 }

\def\y{1}
 \foreach \x in {2,...,5} 
 {
     \draw (\x + 0, \y) -- (\x+1,\y) -- (\x+1,\y+1)-- (\x+0,\y+1) -- cycle;
 } 
 
\def\y{2}
 \foreach \x in {5} 
 {
     \draw (\x + 0, \y) -- (\x+1,\y) -- (\x+1,\y+1)-- (\x+0,\y+1) -- cycle;
 }  

\def\y{3}
\foreach \x in {5,6} 
{
     \draw (\x + 0, \y) -- (\x+1,\y) -- (\x+1,\y+1)-- (\x+0,\y+1) -- cycle;
 } 
 
\end{tikzpicture} 
\newcommand\matchingcolor{2pt}
 \begin{tikzpicture}[scale=\myscale,font=\tiny]
 \foreach \x in {0,1,...,2} 
 {
\def\y{0}
     \ifthenelse{\x = 0 \OR \x =2}
     {
     \macropositivetile
     }{
     \macronegativetile
     }
     \ifthenelse{\x = 0 \OR \x =2}{\draw[line width=\matchingcolor] (\x + 0, \y+0) -- (\x+1,\y+0);}{} 
     \ifthenelse{\x = 0}{\draw[line width=\matchingcolor] (\x + \y+0, \y+1) -- (\x+1,\y+1);}{} 
 }

\def\y{1}
 \foreach \x in {2,...,5} 
 {
     \ifthenelse{\x = 3 \OR \x =5}
     {\macropositivetile}{\macronegativetile}
     \ifthenelse{\x = 2}{\draw[line width=\matchingcolor] (\x+0, \y+0) -- (\x, \y+1);}{} 
     \ifthenelse{\x = 3 \OR \x =5}{\draw[line width=\matchingcolor] (\x+0, \y+0) -- (\x+1, \y+0);}{} 
     \ifthenelse{\x = 3}{\draw[line width=\matchingcolor] (\x+0, \y+1) -- (\x+1,\y+1);}{} 
 } 
 
\def\y{2}
 \foreach \x in {5} 
 {
     \macronegativetile
     \ifthenelse{\x = 5}{\draw[line width=\matchingcolor] (\x+0, \y+0) -- (\x, \y+1);}{} 
     \ifthenelse{\x = 5}{\draw[line width=\matchingcolor] (\x+1, \y) -- (\x+1,\y+1);}{} 
 }  

\def\y{3}
\foreach \x in {5,6} 
{
     \ifthenelse{\x = 5}
     {\macropositivetile}{\macronegativetile}
     \ifthenelse{\x = 5}{\draw[line width=\matchingcolor] (\x+0, \y+1) -- (\x+1,\y+1);}{} 
     \ifthenelse{\x = 6}{\draw[line width=\matchingcolor] (\x+1, \y) -- (\x+1,\y+1);}{} 
 } 
\end{tikzpicture} 
 \begin{tikzpicture}[scale=\myscale,font=\tiny]
 \foreach \x in {0,1,...,2} 
 {
\def\y{0}
     \ifthenelse{\x = 0 \OR \x =2}
     {
     \macropositivetile
     }{
     \macronegativetile
     }
     \ifthenelse{\x = 0 \OR \x =2}{\draw[line width=\matchingcolor] (\x + 0, \y+0) -- (\x+1,\y+0);}{} 
     \ifthenelse{\x = 0}{\draw[line width=\matchingcolor] (\x + \y+0, \y+1) -- (\x+1,\y+1);}{} 
 }

\def\y{1}
 \foreach \x in {2,...,5} 
 {
     \ifthenelse{\x = 3 \OR \x =5}
     {\macropositivetile}{\macronegativetile}
     \ifthenelse{\x = 2 \OR \x =5}{\draw[line width=\matchingcolor] (\x+0, \y+0) -- (\x+1, \y+0);}{} 
     \ifthenelse{\x = 2}{\draw[line width=\matchingcolor] (\x+0, \y+1) -- (\x+1,\y+1);}{} 
      \ifthenelse{\x = 3}{\draw[line width=\matchingcolor] (\x+1, \y) -- (\x+1,\y+1);}{} 
 } 
 
\def\y{2}
 \foreach \x in {5} 
 {
     \macronegativetile
     \ifthenelse{\x = 5}{\draw[line width=\matchingcolor] (\x+0, \y+0) -- (\x+1, \y+0);}{} 
 }  

\def\y{3}
\foreach \x in {5,6} 
{
     \ifthenelse{\x = 5}
     {\macropositivetile}{\macronegativetile}
        \ifthenelse{\x = 5}{\draw[line width=\matchingcolor] (\x+0, \y+0) -- (\x, \y+1);}{} 

     \ifthenelse{\x = 6}{\draw[line width=\matchingcolor] (\x+0, \y+0) -- (\x+1, \y+0);}{} 
     \ifthenelse{\x = 6}{\draw[line width=\matchingcolor] (\x+0, \y+1) -- (\x+1,\y+1);}{} 
 } 
\end{tikzpicture} 
\caption{A snake graph (left); the minimal perfect matching (center); another perfect matching  of the snake graph (right)}\label{fig:snake_graph_matching}
\label{fig:minimalmatching}
\end{center}
\end{figure}
\end{definition}

A cluster algebra~\cite{FZ02} is a commutative algebra with distinguished generators called \defn{cluster variables} which can be written as Laurent polynomials with positive coefficients. In the case of a family of cluster algebras called \defn{cluster algebras from surfaces}, given such a Laurent polynomial $x_\gamma$, it was shown in~\cite[Theorem~4.17]{MSW11} that $x_\gamma$ can be associated to a certain snake graph $G_\gamma$ and that $x_\gamma$ can be written as a sum over all perfect matchings of $G_\gamma$. In particular, the terms of $x_\gamma$ are in bijection with the perfect matchings of $G_\gamma$.

The following allows us to associate a snake graph to a binary word. 

\begin{definition}[{\cite[Section~2.1]{CS13}}]
\label{def:sign_function}
A \defn{sign function} on a snake graph $G$ is a map from the set of edges of G to 
$\{+,-\}$ 
such that, for every tile of $G$, the north edge and the west edge have the same sign, 
 the south edge and the east edge have the same sign, and 
 the sign on the north edge is opposite to the sign on the south edge.

\begin{figure}[ht]
\begin{center}
\def\figexamplesignfunctionscale{1}
\def\signsequencescale{1}
\newcommand\macropositivetile{
     \def\south{$-$}
     \def\east{$-$}
     \def\north{$+$}
     \def\west{$+$} 
     \draw (\x + 0, \y) --  (\x+1,\y) node [pos=0.5,above] {\south} -- (\x+1,\y+1) node [pos=0.5,left] {\east} -- (\x+0,\y+1) node [pos=0.5,above] {\north} --  (\x + 0, \y)  node [pos=0.5,left] {\west};
}
\newcommand\macropositivetileBinary{
     \def\south{$0$}
     \def\east{$0$}
     \def\north{$1$}
     \def\west{$1$} 
     \draw (\x + 0, \y) --  (\x+1,\y) node [pos=0.5,above] {\south} -- (\x+1,\y+1) node [pos=0.5,left] {\east} -- (\x+0,\y+1) node [pos=0.5,above] {\north} --  (\x + 0, \y)  node [pos=0.5,left] {\west};
}
\newcommand\macropositivetileBinarySouthOne{
     \def\south{$1$}
     \def\east{$1$}
     \def\north{$0$}
     \def\west{$0$} 
     \draw (\x + 0, \y) --  (\x+1,\y) node [pos=0.5,above] {\south} -- (\x+1,\y+1) node [pos=0.5,left] {\east} -- (\x+0,\y+1) node [pos=0.5,above] {\north} --  (\x + 0, \y)  node [pos=0.5,left] {\west};
}
\newcommand\macronegativetile{
     \def\south{$+$}
     \def\east{$+$}
     \def\north{$-$}
     \def\west{$-$} 
     \draw (\x + 0, \y) --  (\x+1,\y) node [pos=0.5,above] {\south} -- (\x+1,\y+1) node [pos=0.5,left] {\east} -- (\x+0,\y+1) node [pos=0.5,above] {\north} --  (\x + 0, \y)  node [pos=0.5,left] {\west};
}
\newcommand\macronegativetileBinary{
     \def\south{$1$}
     \def\east{$1$}
     \def\north{$0$}
     \def\west{$0$} 
     \draw (\x + 0, \y) --  (\x+1,\y) node [pos=0.5,above] {\south} -- (\x+1,\y+1) node [pos=0.5,left] {\east} -- (\x+0,\y+1) node [pos=0.5,above] {\north} --  (\x + 0, \y)  node [pos=0.5,left] {\west};
}
\newcommand\macronegativetileBinarySouthOne{
     \def\south{$0$}
     \def\east{$0$}
     \def\north{$1$}
     \def\west{$1$} 
     \draw (\x + 0, \y) --  (\x+1,\y) node [pos=0.5,above] {\south} -- (\x+1,\y+1) node [pos=0.5,left] {\east} -- (\x+0,\y+1) node [pos=0.5,above] {\north} --  (\x + 0, \y)  node [pos=0.5,left] {\west};
}
\begin{tikzpicture}[scale=\figexamplesignfunctionscale,font=\tiny]
 \foreach \x in {0,1,...,2} 
 {
\def\y{0}
     \ifthenelse{\x = 0 \OR \x =2}
     {
     \macropositivetile
     }{
     \macronegativetile
     }
 }

\def\y{1}
 \foreach \x in {2,...,5} 
 {
     \ifthenelse{\x = 3 \OR \x =5}
     {\macropositivetile}{\macronegativetile}

 } 
 
\def\y{2}
 \foreach \x in {5} 
 {
     \macronegativetile
 }  

\def\y{3}
\foreach \x in {5,6} 
{
     \ifthenelse{\x = 5}
     {\macropositivetile}{\macronegativetile}
 } 
\end{tikzpicture}
\hspace{-5mm}
 \begin{tikzpicture}[scale=\figexamplesignfunctionscale,font=\tiny]
 \foreach \x in {0,1,...,2} 
 {
\def\y{0}
     \ifthenelse{\x = 0 \OR \x =2}
     {
     \macropositivetileBinary
     }{
     \macronegativetileBinary
     }
 }

\def\y{1}
 \foreach \x in {2,...,5} 
 {
     \ifthenelse{\x = 3 \OR \x =5}
     {\macropositivetileBinary}{\macronegativetileBinary}

 } 
 
\def\y{2}
 \foreach \x in {5} 
 {
     \macronegativetileBinary
 }  

\def\y{3}
\foreach \x in {5,6} 
{
     \ifthenelse{\x = 5}
     {\macropositivetileBinary}{\macronegativetileBinary}
 } 
\end{tikzpicture}

 \begin{tikzpicture}[scale=\signsequencescale,font=\scriptsize]
\def\macropositivetileBinaryFirstTileWestEast{
     \def\south{}
     \def\east{\textcolor{blue}{$0$}}
     \def\north{}
     \def\west{\textcolor{blue}{\underline{$1$}}}
     \draw (\x + 0, \y) --  (\x+1,\y) node [pos=0.5,above=0pt] {\south} -- (\x+1,\y+1) node [pos=0.5,left] {\east} -- (\x+0,\y+1) node [pos=0.5,above=1mm] {\north} --  (\x + 0, \y)  node [pos=0.5,left=1mm] {\west};
}
\def\macropositivetileBinaryWestEast{
     \def\south{}
     \def\east{\textcolor{blue}{$0$}}
     \def\north{}
     \def\west{\textcolor{blue}{$1$}} 
     \draw (\x + 0, \y) --  (\x+1,\y) node [pos=0.5,above] {\south} -- (\x+1,\y+1) node [pos=0.5,left] {\east} -- (\x+0,\y+1) node [pos=0.5,above] {\north} --  (\x + 0, \y)  node [pos=0.5,left] {\west};
}
\def\macropositivetileBinaryWestNorth{
     \def\south{}
     \def\east{}
     \def\north{\textcolor{blue}{$1$}}
     \def\west{\textcolor{blue}{$1$}} 
     \draw (\x + 0, \y) --  (\x+1,\y) node [pos=0.5,above=0pt] {\south} -- (\x+1,\y+1) node [pos=0.5,left] {\east} -- (\x+0,\y+1) node [pos=0.5,above] {\north} --  (\x + 0, \y)  node [pos=0.5,left] {\west};
}
\def\macropositivetileBinarySouthEast{
     \def\south{\textcolor{blue}{$0$}}
     \def\east{\textcolor{blue}{$0$}}
     \def\north{}
     \def\west{} 
     \draw (\x + 0, \y) --  (\x+1,\y) node [pos=0.5,above] {\south} -- (\x+1,\y+1) node [pos=0.5,left] {\east} -- (\x+0,\y+1) node [pos=0.5,above] {\north} --  (\x + 0, \y)  node [pos=0.5,left] {\west};
}
\def\macropositivetileBinarySouthNorth{
     \def\south{\textcolor{blue}{$0$}}
     \def\east{}
     \def\north{\textcolor{blue}{$1$}}
     \def\west{} 
     \draw (\x + 0, \y) --  (\x+1,\y) node [pos=0.5,above] {\south} -- (\x+1,\y+1) node [pos=0.5,left] {\east} -- (\x+0,\y+1) node [pos=0.5,above] {\north} --  (\x + 0, \y)  node [pos=0.5,left] {\west};
}
\def\macronegativetileBinary{
     \def\south{$1$}
     \def\east{$1$}
     \def\north{$0$}
     \def\west{$0$} 
     \draw (\x + 0, \y) --  (\x+1,\y) node [pos=0.5,above] {\south} -- (\x+1,\y+1) node [pos=0.5,left] {\east} -- (\x+0,\y+1) node [pos=0.5,above] {\north} --  (\x + 0, \y)  node [pos=0.5,left] {\west};
}
\def\macronegativetileBinaryWestEast{
     \def\south{}
     \def\east{\textcolor{blue}{$1$}}
     \def\north{}
     \def\west{\textcolor{blue}{$0$}}
     \draw (\x + 0, \y) --  (\x+1,\y) node [pos=0.5,above] {\south} -- (\x+1,\y+1) node [pos=0.5,left] {\east} -- (\x+0,\y+1) node [pos=0.5,above] {\north} --  (\x + 0, \y)  node [pos=0.5,left] {\west};
}
\def\macronegativetileBinaryWestNorth{
     \def\south{}
     \def\east{}
     \def\north{\textcolor{blue}{$0$}}
     \def\west{\textcolor{blue}{$0$}} 
     \draw (\x + 0, \y) --  (\x+1,\y) node [pos=0.5,above] {\south} -- (\x+1,\y+1) node [pos=0.5,left] {\east} -- (\x+0,\y+1) node [pos=0.5,above] {\north} --  (\x + 0, \y)  node [pos=0.5,left] {\west};
}
\def\macronegativetileBinarySouthEast{
     \def\south{\textcolor{blue}{$1$}}
     \def\east{\textcolor{blue}{$1$}}
     \def\north{}
     \def\west{} 
     \draw (\x + 0, \y) --  (\x+1,\y) node [pos=0.5,above] {\south} -- (\x+1,\y+1) node [pos=0.5,left] {\east} -- (\x+0,\y+1) node [pos=0.5,above] {\north} --  (\x + 0, \y)  node [pos=0.5,left] {\west};
}
\def\macronegativetileBinarySouthNorth{
     \def\south{\textcolor{blue}{$1$}}
     \def\east{}
     \def\north{\textcolor{blue}{$0$}}
     \def\west{} 
     \draw (\x + 0, \y) --  (\x+1,\y) node [pos=0.5,above] {\south} -- (\x+1,\y+1) node [pos=0.5,left] {\east} -- (\x+0,\y+1) node [pos=0.5,above] {\north} --  (\x + 0, \y)  node [pos=0.5,left] {\west};
}
\def\macronegativetileBinaryWest{ 
     \def\south{}
     \def\east{}
     \def\north{}
     \def\west{\textcolor{blue}{$0$}} 
     \draw (\x + 0, \y) --  (\x+1,\y) node [pos=0.5,above] {\south} -- (\x+1,\y+1) node [pos=0.5,left] {\east} -- (\x+0,\y+1) node [pos=0.5,above] {\north} --  (\x + 0, \y)  node [pos=0.5,left] {\west};
}

 \foreach \x in {0,1,...,2} 
 {
\def\y{0}
     \ifthenelse{\x = 0}{\macropositivetileBinaryFirstTileWestEast}{}
     \ifthenelse{\x = 1}{\macronegativetileBinaryWestEast}{}
     \ifthenelse{\x =2}{\macropositivetileBinaryWestNorth}{}
 }

\def\y{1}
 \foreach \x in {2,...,5} 
 {
     \ifthenelse{\x = 2}
     {\macronegativetileBinarySouthEast}{}
     \ifthenelse{\x = 3}
     {\macropositivetileBinaryWestEast}{}
     \ifthenelse{\x = 4}
     {\macronegativetileBinaryWestEast}{}
     \ifthenelse{\x = 5}
     {\macropositivetileBinaryWestNorth}{}
 } 
 
\def\y{2}
 \foreach \x in {5} 
 {
     \macronegativetileBinarySouthNorth
 }  

\def\y{3}
\foreach \x in {5,6} 
{
     \ifthenelse{\x = 5}{\macropositivetileBinarySouthEast}{}
     \ifthenelse{\x = 6}{\macronegativetileBinaryWest}{}
 } 
 
\end{tikzpicture}
\caption{
Sign function 
(top) and sign sequence (bottom) of the snake graph corresponding to the binary expansion $\textcolor{blue}{\underline{1}011101100}$}
\label{fig:example:sign_function}
\label{fig:signsequence}
\end{center}
\end{figure}
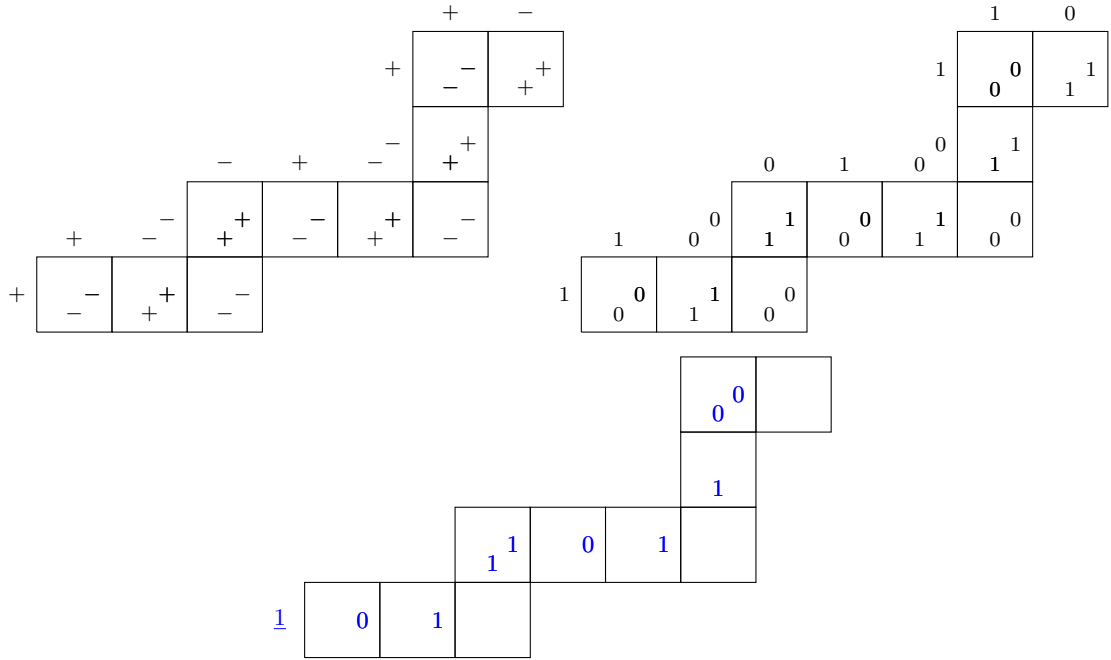
\end{definition}

Note that there are exactly two sign functions on every snake graph. 
We consider only the sign function where the south edge of the first tile $G_1$ has label $-$, see Fig.~\ref{fig:example:sign_function} (top left). 
Since we study binary expansions, we replace $+$ with $1$ and $-$ with $0$, see Fig.~\ref{fig:example:sign_function} (top right).

Given the sign function of a snake graph $G$ whose west edge of the first tile $G_1$ 
 has sign~$1$, 
let the \defn{sign sequence} of $G$ be the sequence $(1,w_2,\dots,w_d)$ where 
$w_2, \dots, w_d$ 
are the signs of the interior edges of the snake graph, see Fig.~\ref{fig:signsequence} (bottom). 
As this sequence uniquely determines a snake graph, we can associate to each binary word $w=1w_2\dots w_d$ a snake graph $G(w)$. 

\subsection{Bijection from subwords to perfect matchings}
An \defn{order filter} (or \defn{upper order ideal}) 
is 
a subset $F$ of $P$ such that if $t \in F$ and $s \geq t$, then $s \in F$.
The set of perfect matchings of a snake graph $G$ is known to form a distributive lattice 
which is 
 isomorphic to the lattice of order filters of the piecewise-linear poset corresponding to $G$~ \cite[Section~5]{MSW13}
 (as a consequence of~\cite[Theorem~2]{Pro02}, see also~\cite[Theorem~3]{Fel04}).
It is also known that the map which sends an order filter to its set of minimal elements 
is a bijection between the order filters and the antichains of a poset. 
Therefore, by Theorem~\ref{thm:antichain_to_subword}, there is a bijection from the subwords of $w$ to the perfect matchings of $G(w)$.

\begin{theorem}
\label{thm:subwords_to_snake_graph_matchings}
Given a binary word $w$ and its corresponding snake graph $G=G(w)$,  
the following map $pm$ from the subwords of $w$ to the perfect matchings of $G$ is a bijection. 
\begin{enumerate}[(step 1)]
\item 
Let $s$ be a subword of $w$. 
If $s$ is the empty word, let $pm(s)$ be $pm_{\text{min}}$. 
Otherwise, write $s=w_{i_1}w_{i_2} \dots w_{i_\ell}$ in such a way that each index $i_k$ is as small as possible (as we do in Section~\ref{sec:bijection_antichains_subwords}) and 
circle the edges of $G$ corresponding to the sign sequence for $s$.

\item \label{itm:step2}
For each block $L$ 
of consecutive circled edges, let $\square_L$ be the tile which is immediately north/east of the last edge in $L$. 

\item \label{itm:step3}
Let $fil(\square_L)$ be 
the smallest connected sequence of tiles such that $\square_L \in fil(\square_L)$ and the set of edges bounding $fil(\square_L)$ not in $pm_{\text{min}}$ forms a perfect matching of $fil(\square_L)$. 
Set $ fil(s)= \bigcup_{L} fil(\square_L)$.

\item 
Define $pm(s)$ to be the symmetric difference $\{$edges bounding $fil(s)\} \ominus pm_{\text{min}}$.  
\end{enumerate}
\end{theorem}

Let $P$ be the piecewise-linear poset associated to $w$, and associate the $i$th tile of $G$ with the $i$th element of $P$. 
In step~\ref{itm:step2}, the tiles $\square_L$ correspond to the antichain $A$ of $P$ coming from $s$. 
The tiles in $fil(s)$ in step~\ref{itm:step3} correspond to the order filter of $P$ generated by $A$.
In the following example, 
we 
construct the 
perfect matchings corresponding to 
 the two antichains in Fig.~\ref{fig:example:antichain_to_subword}.

\def\myscalematching{0.6}
\begin{figure}[htb]
	\begin{center}
 \begin{tikzpicture}[ xscale=0.65, yscale=0.65, 
 font=\scriptsize]
\newcommand\macrotile[1]{
     \def\south{}
     \def\east{}
     \def\north{}
     \def\west{} 
     \draw (\x + 0, \y) --  (\x+1,\y) node [pos=0.5,above] {\south} -- (\x+1,\y+1) node [pos=0.5,left] {\east} -- (\x+0,\y+1) node [pos=0.5,above] {\north} --  (\x + 0, \y)  node [pos=0.5,left] {\west};
}
\newcommand\macropositivetileBinary{
     \def\south{\textcolor{black}{$0$}}
     \def\east{\textcolor{black}{$0$}}
     \def\north{\textcolor{black}{$1$}}
     \def\west{\textcolor{black}{$1$}}
     \draw[gray] (\x + 0, \y) --  (\x+1,\y) node [pos=0.5,above] {\south} -- (\x+1,\y+1) node [pos=0.5,left] {\east} -- (\x+0,\y+1) node [pos=0.5,above] {\north} --  (\x + 0, \y)  node [pos=0.5,left] {\west};
}
\newcommand\macropositivetileBinaryWestEast{
     \def\south{}
     \def\east{\textcolor{blue}{\circled{$0$}}}
     \def\north{}
     \def\west{\textcolor{blue}{\circled{$1$}}}
     \draw (\x + 0, \y) --  (\x+1,\y) node [pos=0.5] {\south} -- (\x+1,\y+1) node [pos=0.5] {\east} -- (\x+0,\y+1) node [pos=0.5] {\north} --  (\x + 0, \y)  node [pos=0.5] {\west};
}
\newcommand\macropositivetileBinaryEast{
     \def\south{}
     \def\east{\textcolor{blue}{\circled{$0$}}}
     \def\north{}
     \def\west{{$1$}} 
     \draw (\x + 0, \y) --  (\x+1,\y) node [pos=0.5] {\south} -- (\x+1,\y+1) node [pos=0.5] {\east} -- (\x+0,\y+1) node [pos=0.5] {\north} --  (\x + 0, \y)  node [pos=0.5] {\west};
}
\newcommand\macropositivetileBinaryWestNorth{
     \def\south{}
     \def\east{}
     \def\north{\textcolor{blue}{\circled{$1$}}}
     \def\west{\textcolor{blue}{\circled{$1$}}} 
     \draw (\x + 0, \y) --  (\x+1,\y) node [pos=0.5] {\south} -- (\x+1,\y+1) node [pos=0.5] {\east} -- (\x+0,\y+1) node [pos=0.5] {\north} --  (\x + 0, \y)  node [pos=0.5] {\west};
}
\newcommand\macropositivetileBinarySouthEast{
     \def\south{\textcolor{blue}{\circled{$0$}}}
     \def\east{\textcolor{blue}{\circled{$0$}}}
     \def\north{}
     \def\west{} 
     \draw (\x + 0, \y) --  (\x+1,\y) node [pos=0.5] {\south} -- (\x+1,\y+1) node [pos=0.5] {\east} -- (\x+0,\y+1) node [pos=0.5] {\north} --  (\x + 0, \y)  node [pos=0.5] {\west};
}
\newcommand\macropositivetileBinarySouthNorth{
     \def\south{\textcolor{blue}{\circled{$0$}}}
     \def\east{}
     \def\north{\textcolor{blue}{\circled{$1$}}}
     \def\west{} 
     \draw (\x + 0, \y) --  (\x+1,\y) node [pos=0.5] {\south} -- (\x+1,\y+1) node [pos=0.5] {\east} -- (\x+0,\y+1) node [pos=0.5] {\north} --  (\x + 0, \y)  node [pos=0.5] {\west};
}

\newcommand\macronegativetileBinary{
     \def\south{\textcolor{black}{$1$}}
     \def\east{\textcolor{black}{$1$}}
     \def\north{\textcolor{black}{$0$}}
     \def\west{\textcolor{black}{$0$}}
     \draw (\x + 0, \y) --  (\x+1,\y) node [pos=0.5] {\south} -- (\x+1,\y+1) node [pos=0.5] {\east} -- (\x+0,\y+1) node [pos=0.5] {\north} --  (\x + 0, \y)  node [pos=0.5] {\west};
}
\newcommand\macronegativetileBinaryWestEast{
     \def\south{}
     \def\east{\textcolor{blue}{\circled{$1$}}}
     \def\north{}
     \def\west{\textcolor{blue}{\circled{$0$}}}
     \draw (\x + 0, \y) --  (\x+1,\y) node [pos=0.5] {\south} -- (\x+1,\y+1) node [pos=0.5] {\east} -- (\x+0,\y+1) node [pos=0.5] {\north} --  (\x + 0, \y)  node [pos=0.5] {\west};
}
\newcommand\macronegativetileBinaryWestNorth{
     \def\south{}
     \def\east{}
     \def\north{\textcolor{blue}{\circled{$0$}}}
     \def\west{\textcolor{blue}{\circled{$0$}}}
     \draw (\x + 0, \y) --  (\x+1,\y) node [pos=0.5] {\south} -- (\x+1,\y+1) node [pos=0.5] {\east} -- (\x+0,\y+1) node [pos=0.5] {\north} --  (\x + 0, \y)  node [pos=0.5] {\west};
}
\newcommand\macronegativetileBinarySouthEast{
     \def\south{\textcolor{blue}{\circled{$1$}}}
     \def\east{\textcolor{blue}{\circled{$1$}}}
     \def\north{}
     \def\west{} 
     \draw (\x + 0, \y) --  (\x+1,\y) node [pos=0.5,above] {\south} -- (\x+1,\y+1) node [pos=0.5,left] {\east} -- (\x+0,\y+1) node [pos=0.5,above] {\north} --  (\x + 0, \y)  node [pos=0.5,left] {\west};
}

\newcommand\macronegativetileBinarySouthNorth{
     \def\south{\textcolor{blue}{\circled{$1$}}}
     \def\east{}
     \def\north{\textcolor{blue}{\circled{$0$}}}
     \def\west{} 
     \draw (\x + 0, \y) --  (\x+1,\y) node [pos=0.5] {\south} -- (\x+1,\y+1) node [pos=0.5] {\east} -- (\x+0,\y+1) node [pos=0.5] {\north} --  (\x + 0, \y)  node [pos=0.5] {\west};
}
\newcommand\macronegativetileBinaryWest{ 
     \def\south{}
     \def\east{}
     \def\north{}
     \def\west{\textcolor{blue}{\circled{$0$}}}
     \draw (\x + 0, \y) --  (\x+1,\y) node [pos=0.5] {\south} -- (\x+1,\y+1) node [pos=0.5] {\east} -- (\x+0,\y+1) node [pos=0.5] {\north} --  (\x + 0, \y)  node [pos=0.5] {\west};
}

 \foreach \x in {0,1,...,2} 
 {
\def\y{0}
     \ifthenelse{\x = 0}{\macropositivetileBinaryWestEast}{}
     \ifthenelse{\x = 1}{\macronegativetileBinaryWestEast}{}
     \ifthenelse{\x =2}{\macropositivetileBinaryWestNorth}{}
 }

\def\y{1}
 \foreach \x in {2,...,5} 
 {
     \ifthenelse{\x = 2}
     {\macrotile{} \draw (\x + 0.5, \y+0.5) node {\boxed{$d$}}; }{}
     \ifthenelse{\x = 3}
     {\macropositivetileBinaryEast}{}
     \ifthenelse{\x = 4}
     {\macronegativetileBinaryWestEast}{}
     \ifthenelse{\x = 5}
     {\macropositivetileBinaryWestNorth}{}
 } 
 
\def\y{2}
 \foreach \x in {5} 
 {
     \macronegativetileBinarySouthNorth
 }  

\def\y{3}
\foreach \x in {5,6} 
{
     \ifthenelse{\x = 5}{\macropositivetileBinarySouthEast}{}
     \ifthenelse{\x = 6}{\macronegativetileBinaryWest \draw (\x + 0.5, \y+0.5) node {$\boxed{j}$};}{}
 } 
 
\end{tikzpicture}
		\hspace{-2.5cm}
		$\square_{\textcolor{blue}{1011}}$=$\boxed{d}$ and $\square_{\textcolor{blue}{01100}}$=$\boxed{j}$ 
		\hspace{-0.5cm}
\def\macrofilledtile{
\draw [fill=lightgray,opacity=0.99]  (\x + 0, \y) rectangle (\x+1,\y+1);
}
\def\macropositivetile{
     \def\south{}
     \def\east{}
     \def\north{}
     \def\west{} 
     \draw (\x + 0, \y) --  (\x+1,\y) node [pos=0.5,above] {\south} -- (\x+1,\y+1) node [pos=0.5,left] {\east} -- (\x+0,\y+1) node [pos=0.5,above] {\north} --  (\x + 0, \y)  node [pos=0.5,left] {\west};
}
\def\macronegativetile{
     \def\south{}
     \def\east{}
     \def\north{}
     \def\west{} 
     \draw (\x + 0, \y) --  (\x+1,\y) node [pos=0.5,above] {\south} -- (\x+1,\y+1) node [pos=0.5,left] {\east} -- (\x+0,\y+1) node [pos=0.5,above] {\north} --  (\x + 0, \y)  node [pos=0.5,left] {\west};
}
\def\matchingcolor{2pt}
 \begin{tikzpicture}[scale=\myscalematching,font=\footnotesize]
 \foreach \x in {0,1,...,2} 
 {
\def\y{0}
     \ifthenelse{\x = 0 \OR \x =2}
     {
     \macropositivetile
     }{
     \macronegativetile
     }
     \ifthenelse{\x = 0 \OR \x =2}{\draw[line width=\matchingcolor] (\x + 0, \y+0) -- (\x+1,\y+0);}{} 
     \ifthenelse{\x = 0}{\draw[line width=\matchingcolor] (\x + \y+0, \y+1) -- (\x+1,\y+1);}{} 
 }
\def\y{1}
 \foreach \x in {2,...,5} 
 {
     \ifthenelse{\x = 2 \OR \x = 3}{\macrofilledtile}{}
     \ifthenelse{\x = 2}
     {\draw (\x + 0.5, \y+0.5) node {\boxed{$d$}}; }{}
     \ifthenelse{\x = 3}
     {\draw (\x + 0.5, \y+0.5) node {{$e$}}; }{}
     \ifthenelse{\x = 3 \OR \x =5}
     {\macropositivetile}{\macronegativetile}
     \ifthenelse{\x = 2}{\draw[line width=\matchingcolor] (\x+0, \y+0) -- (\x, \y+1);}{} 
     \ifthenelse{\x = 2}{\draw[red, densely dashed, line width=1.5*\matchingcolor] (\x+0, \y+0) -- (\x+1, \y+0);}{} 
     \ifthenelse{\x = 2}{\draw[red, densely dashed, line width=1.5*\matchingcolor] (\x+0, \y+1) -- (\x+1,\y+1);}{} 
      \ifthenelse{\x = 3}{\draw[red, densely dashed, line width=1.5*\matchingcolor] (\x+1, \y) -- (\x+1,\y+1);}{} 
      \ifthenelse{\x = 3}{\draw[line width=\matchingcolor] (\x+0, \y+0) -- (\x+1, \y+0);}{} 
     \ifthenelse{\x = 3}{\draw[line width=\matchingcolor] (\x+0, \y+1) -- (\x+1,\y+1);}{} 
      \ifthenelse{\x =5}{\draw[line width=\matchingcolor] (\x+0, \y+0) -- (\x+1, \y+0);}{} 
 } 
\def\y{2}
 \foreach \x in {5} 
 {
     \macrofilledtile
     \macronegativetile
     \ifthenelse{\x = 5}{\draw[line width=\matchingcolor] (\x+0, \y+0) -- (\x, \y+1);}{} 
     \ifthenelse{\x = 5}{\draw[line width=\matchingcolor] (\x+1, \y) -- (\x+1,\y+1);}{} 
     \ifthenelse{\x = 5}{\draw[red,densely dashed,line width=1.5*\matchingcolor] (\x+0, \y+0) -- (\x+1, \y+0);}{} 
     \draw (\x + 0.5, \y+0.5) node {{$h$}};
 }  
\def\y{3}
\foreach \x in {5,6} 
{
    \macrofilledtile
     \ifthenelse{\x = 5}
     {\macropositivetile}{\macronegativetile}
     \ifthenelse{\x = 5}{\draw[line width=\matchingcolor] (\x+0, \y+1) -- (\x+1,\y+1);}{} 
     \ifthenelse{\x = 5}{\draw[red,densely dashed,line width=1.5*\matchingcolor] (\x+0, \y+0) -- (\x, \y+1);}{} 

     \ifthenelse{\x = 6}{\draw[line width=\matchingcolor] (\x+1, \y) -- (\x+1,\y+1);}{} 
     \ifthenelse{\x = 6}{\draw[red,densely dashed,line width=1.5*\matchingcolor] (\x+0, \y+0) -- (\x+1, \y+0);}{} 
     \ifthenelse{\x = 6}{\draw[red,densely dashed,line width=1.5*\matchingcolor] (\x+0, \y+1) -- (\x+1,\y+1);}{} 
     \ifthenelse{\x = 5}{\draw (\x + 0.5, \y+0.5) node {{$i$}}; }{}
     \ifthenelse{\x = 6}{\draw (\x + 0.5, \y+0.5) node {\boxed{$j$}}; }{}
 } 
\end{tikzpicture} 
		\hspace{-1.4cm}
\def\macrofilledtile{
\draw [fill=lightgray,opacity=0.99]  (\x + 0, \y) rectangle (\x+1,\y+1);
}
\def\macropositivetile{
     \def\south{}
     \def\east{}
     \def\north{}
     \def\west{} 
     \draw (\x + 0, \y) --  (\x+1,\y) node [pos=0.5,above] {\south} -- (\x+1,\y+1) node [pos=0.5,left] {\east} -- (\x+0,\y+1) node [pos=0.5,above] {\north} --  (\x + 0, \y)  node [pos=0.5,left] {\west};
}
\def\macronegativetile{
     \def\south{}
     \def\east{}
     \def\north{}
     \def\west{} 
     \draw (\x + 0, \y) --  (\x+1,\y) node [pos=0.5,above] {\south} -- (\x+1,\y+1) node [pos=0.5,left] {\east} -- (\x+0,\y+1) node [pos=0.5,above] {\north} --  (\x + 0, \y)  node [pos=0.5,left] {\west};
}
\def\matchingcolor{2pt}
 \begin{tikzpicture}[scale=\myscalematching,font=\footnotesize]
 \foreach \x in {0,1,...,2} 
 {
\def\y{0}
     \ifthenelse{\x = 0 \OR \x =2}
     {
     \macropositivetile
     }{
     \macronegativetile
     }
     \ifthenelse{\x = 0 \OR \x =2}{\draw[line width=\matchingcolor] (\x + 0, \y+0) -- (\x+1,\y+0);}{} 
     \ifthenelse{\x = 0}{\draw[line width=\matchingcolor] (\x + \y+0, \y+1) -- (\x+1,\y+1);}{} 
 }
\def\y{1}
 \foreach \x in {2,...,5} 
 {
     \ifthenelse{\x = 2 \OR \x = 3}{\macrofilledtile}{}
     \ifthenelse{\x = 3 \OR \x =5}{\macropositivetile}{\macronegativetile}
     \ifthenelse{\x = 2}{\draw[red, 
     line width=1.5*\matchingcolor] (\x+0, \y+0) -- (\x+1, \y+0);}{} 
     \ifthenelse{\x = 2}{\draw[red, 
     line width=1.5*\matchingcolor] (\x+0, \y+1) -- (\x+1,\y+1);}{} 
      \ifthenelse{\x = 3}{\draw[red, 
      line width=1.5*\matchingcolor] (\x+1, \y) -- (\x+1,\y+1);}{} 
      \ifthenelse{\x =5}{\draw[line width=\matchingcolor] (\x+0, \y+0) -- (\x+1, \y+0);}{} 
 } 
\def\y{2}
 \foreach \x in {5} 
 {
     \macrofilledtile
     \macronegativetile
     \ifthenelse{\x = 5}{\draw[red,
     line width=1.5*\matchingcolor] (\x+0, \y+0) -- (\x+1, \y+0);}{} 
 }  
\def\y{3}
\foreach \x in {5,6} 
{
    \macrofilledtile
     \ifthenelse{\x = 5}
     {\macropositivetile}{\macronegativetile}
     \ifthenelse{\x = 5}{\draw[red,
     line width=1.5*\matchingcolor] (\x+0, \y+0) -- (\x, \y+1);}{} 

     \ifthenelse{\x = 6}{\draw[red,
     line width=1.5*\matchingcolor] (\x+0, \y+0) -- (\x+1, \y+0);}{} 
     \ifthenelse{\x = 6}{\draw[red,
     line width=1.5*\matchingcolor] (\x+0, \y+1) -- (\x+1,\y+1);}{} 
 } 
\end{tikzpicture} 
		\caption{Tiles $\square_L$ associated to blocks $L$ of circled edges (left); 
			the set $fil(s)$ of shaded tiles and the set $pm(s)$  of thick solid edges (right) for the subword 
			s=\textcolor{blue}{$\boxed{1011}\boxed{01100}$} 
		}
		\label{fig:square_L:s101101100}
		\label{fig:example:pm_dk}
 \begin{tikzpicture}[xscale=0.65, yscale=0.65,font=\scriptsize]
\newcommand\macrotile[1]{
     \def\south{}
     \def\east{}
     \def\north{}
     \def\west{} 
     \draw (\x + 0, \y) --  (\x+1,\y) node [pos=0.5,above] {\south} -- (\x+1,\y+1) node [pos=0.5,left] {\east} -- (\x+0,\y+1) node [pos=0.5,above] {\north} --  (\x + 0, \y)  node [pos=0.5,left] {\west};
}
\newcommand\macropositivetileBinaryWestEast{
     \def\south{}
     \def\east{\textcolor{blue}{\circled{$0$}}}
     \def\north{}
     \def\west{\textcolor{blue}{\circled{$1$}}}
     \draw (\x + 0, \y) --  (\x+1,\y) node [pos=0.5] {\south} -- (\x+1,\y+1) node [pos=0.5] {\east} -- (\x+0,\y+1) node [pos=0.5] {\north} --  (\x + 0, \y)  node [pos=0.5] {\west};
}
\newcommand\macropositivetileBinaryWest{
     \def\south{}
     \def\east{{$0$}}
     \def\north{}
     \def\west{\textcolor{blue}{\circled{$1$}}}
     \draw (\x + 0, \y) --  (\x+1,\y) node [pos=0.5] {\south} -- (\x+1,\y+1) node [pos=0.5] {\east} -- (\x+0,\y+1) node [pos=0.5] {\north} --  (\x + 0, \y)  node [pos=0.5] {\west};
}
\newcommand\macropositivetileBinaryEast{
     \def\south{}
     \def\east{\textcolor{blue}{\circled{$0$}}}
     \def\north{}
     \def\west{{$1$}} 
     \draw (\x + 0, \y) --  (\x+1,\y) node [pos=0.5] {\south} -- (\x+1,\y+1) node [pos=0.5] {\east} -- (\x+0,\y+1) node [pos=0.5] {\north} --  (\x + 0, \y)  node [pos=0.5] {\west};
}
\newcommand\macropositivetileBinaryWestNorth{
     \def\south{}
     \def\east{}
     \def\north{\textcolor{blue}{\circled{$1$}}}
     \def\west{\textcolor{blue}{\circled{$1$}}} 
     \draw (\x + 0, \y) --  (\x+1,\y) node [pos=0.5] {\south} -- (\x+1,\y+1) node [pos=0.5] {\east} -- (\x+0,\y+1) node [pos=0.5] {\north} --  (\x + 0, \y)  node [pos=0.5] {\west};
}
\newcommand\macropositivetileBinarySouthEast{
     \def\south{\textcolor{blue}{\circled{$0$}}}
     \def\east{\textcolor{blue}{\circled{$0$}}}
     \def\north{}
     \def\west{} 
     \draw (\x + 0, \y) --  (\x+1,\y) node [pos=0.5] {\south} -- (\x+1,\y+1) node [pos=0.5] {\east} -- (\x+0,\y+1) node [pos=0.5] {\north} --  (\x + 0, \y)  node [pos=0.5] {\west};
}
\newcommand\macropositivetileBinarySouth{
     \def\south{\textcolor{blue}{\circled{$0$}}}
     \def\east{}
     \def\north{}
     \def\west{} 
     \draw (\x + 0, \y) --  (\x+1,\y) node [pos=0.5] {\south} -- (\x+1,\y+1) node [pos=0.5] {\east} -- (\x+0,\y+1) node [pos=0.5] {\north} --  (\x + 0, \y)  node [pos=0.5] {\west};
}
\newcommand\macropositivetileBinarySouthNorth{
     \def\south{\textcolor{blue}{\circled{$0$}}}
     \def\east{}
     \def\north{\textcolor{blue}{\circled{$1$}}}
     \def\west{} 
     \draw (\x + 0, \y) --  (\x+1,\y) node [pos=0.5] {\south} -- (\x+1,\y+1) node [pos=0.5] {\east} -- (\x+0,\y+1) node [pos=0.5] {\north} --  (\x + 0, \y)  node [pos=0.5] {\west};
}

\newcommand\macronegativetileBinary{
     \def\south{$1$}
     \def\east{$1$}
     \def\north{$0$}
     \def\west{$0$} 
     \draw (\x + 0, \y) --  (\x+1,\y) node [pos=0.5] {\south} -- (\x+1,\y+1) node [pos=0.5] {\east} -- (\x+0,\y+1) node [pos=0.5] {\north} --  (\x + 0, \y)  node [pos=0.5] {\west};
}
\newcommand\macronegativetileBinaryWestEast{
     \def\south{}
     \def\east{\textcolor{blue}{\circled{$1$}}}
     \def\north{}
     \def\west{\textcolor{blue}{\circled{$0$}}}
     \draw (\x + 0, \y) --  (\x+1,\y) node [pos=0.5] {\south} -- (\x+1,\y+1) node [pos=0.5] {\east} -- (\x+0,\y+1) node [pos=0.5] {\north} --  (\x + 0, \y)  node [pos=0.5] {\west};
}
\newcommand\macronegativetileBinaryEast{
     \def\south{}
     \def\east{\textcolor{blue}{\circled{$1$}}}
     \def\north{}
     \def\west{{$0$}}
     \draw (\x + 0, \y) --  (\x+1,\y) node [pos=0.5] {\south} -- (\x+1,\y+1) node [pos=0.5] {\east} -- (\x+0,\y+1) node [pos=0.5] {\north} --  (\x + 0, \y)  node [pos=0.5] {\west};
}
\newcommand\macronegativetileBinaryWestNorth{
     \def\south{}
     \def\east{}
     \def\north{\textcolor{blue}{\circled{$0$}}}
     \def\west{\textcolor{blue}{\circled{$0$}}}
     \draw (\x + 0, \y) --  (\x+1,\y) node [pos=0.5] {\south} -- (\x+1,\y+1) node [pos=0.5] {\east} -- (\x+0,\y+1) node [pos=0.5] {\north} --  (\x + 0, \y)  node [pos=0.5] {\west};
}
\newcommand\macronegativetileBinarySouthEast{
     \def\south{\textcolor{blue}{\circled{$1$}}}
     \def\east{\textcolor{blue}{\circled{$1$}}}
     \def\north{}
     \def\west{} 
     \draw (\x + 0, \y) --  (\x+1,\y) node [pos=0.5,above] {\south} -- (\x+1,\y+1) node [pos=0.5,left] {\east} -- (\x+0,\y+1) node [pos=0.5,above] {\north} --  (\x + 0, \y)  node [pos=0.5,left] {\west};
}
\newcommand\macronegativetileBinarySouthEastNotcircled{
     \def\south{\textcolor{black}{{$1$}}}
     \def\east{\textcolor{black}{{$1$}}}
     \def\north{}
     \def\west{} 
     \draw (\x + 0, \y) --  (\x+1,\y) node [pos=0.5,above] {\south} -- (\x+1,\y+1) node [pos=0.5] {\east} -- (\x+0,\y+1) node [pos=0.5,above] {\north} --  (\x + 0, \y)  node [pos=0.5] {\west};
}

\newcommand\macronegativetileBinarySouthNorth{
     \def\south{\textcolor{blue}{\circled{$1$}}}
     \def\east{}
     \def\north{\textcolor{blue}{\circled{$0$}}}
     \def\west{} 
     \draw (\x + 0, \y) --  (\x+1,\y) node [pos=0.5] {\south} -- (\x+1,\y+1) node [pos=0.5] {\east} -- (\x+0,\y+1) node [pos=0.5] {\north} --  (\x + 0, \y)  node [pos=0.5] {\west};
}
\newcommand\macronegativetileBinaryNorthNotcircled{
     \def\south{\textcolor{black}{{$1$}}}
     \def\east{}
     \def\north{\textcolor{black}{{$0$}}}
     \def\west{} 
     \draw (\x + 0, \y) --  (\x+1,\y) node [pos=0.5] {\south} -- (\x+1,\y+1) node [pos=0.5] {\east} -- (\x+0,\y+1) node [pos=0.5] {\north} --  (\x + 0, \y)  node [pos=0.5] {\west};
}
\newcommand\macronegativetileBinaryWest{ 
     \def\south{}
     \def\east{}
     \def\north{}
     \def\west{\textcolor{blue}{\circled{$0$}}}
     \draw (\x + 0, \y) --  (\x+1,\y) node [pos=0.5] {\south} -- (\x+1,\y+1) node [pos=0.5] {\east} -- (\x+0,\y+1) node [pos=0.5] {\north} --  (\x + 0, \y)  node [pos=0.5] {\west};
}
\newcommand\macronegativetileBinaryWestNotcircled{ 
     \def\south{}
     \def\east{}
     \def\north{}
     \def\west{\textcolor{black}{{$0$}}}
     \draw (\x + 0, \y) --  (\x+1,\y) node [pos=0.5] {\south} -- (\x+1,\y+1) node [pos=0.5] {\east} -- (\x+0,\y+1) node [pos=0.5] {\north} --  (\x + 0, \y)  node [pos=0.5] {\west};
}

 \foreach \x in {0,1,...,2} 
 {
\def\y{0}
     \ifthenelse{\x = 0}{\macropositivetileBinaryWest      \draw (\x + 0.5, \y+0.5) node {$\boxed{a}$};}{}
     \ifthenelse{\x = 1}{\macronegativetileBinaryEast}{}
     \ifthenelse{\x =2}{\macrotile{}
     \draw (\x + 0.5, \y+0.5) node {$\boxed{c}$};}{}
 }

\def\y{1}
 \foreach \x in {2,...,5} 
 {
     \ifthenelse{\x = 2}
     {\macronegativetileBinarySouthEastNotcircled }{}
     \ifthenelse{\x = 3}
     {\macropositivetileBinaryEast}{}
     \ifthenelse{\x = 4}
     {\macronegativetileBinaryWestEast}{}
     \ifthenelse{\x = 5}
     {\macrotile{}
     \draw (\x + 0.5, \y+0.5) node {$\boxed{g}$};
     }{}
 } 
 
\def\y{2}
 \foreach \x in {5} 
 {
     \macronegativetileBinaryNorthNotcircled{}
 }  

\def\y{3}
\foreach \x in {5,6} 
{
     \ifthenelse{\x = 5}{\macropositivetileBinarySouth
     \draw (\x + 0.5, \y+0.5) node {$\boxed{i}$};}{}
     \ifthenelse{\x = 6}{\macronegativetileBinaryWestNotcircled }{}
 } 
 
\end{tikzpicture}
		\hspace{-34mm}
		{\small $\square_{\textcolor{blue}{1}}$=$\boxed{a}$, $\square_{\textcolor{blue}{1}}$=$\boxed{c}$ , $\square_{\textcolor{blue}{01}}$=$\boxed{g}$, $\square_{\textcolor{blue}{0}}$=$\boxed{i}$}
		\hspace{-7mm}
\def\macrofilledtile{
\draw [fill=lightgray,opacity=0.99]  (\x + 0, \y) rectangle (\x+1,\y+1);
}
\def\macropositivetile{
     \def\south{}
     \def\east{}
     \def\north{}
     \def\west{} 
     \draw (\x + 0, \y) --  (\x+1,\y) node [pos=0.5,above] {\south} -- (\x+1,\y+1) node [pos=0.5,left] {\east} -- (\x+0,\y+1) node [pos=0.5,above] {\north} --  (\x + 0, \y)  node [pos=0.5,left] {\west};
}
\def\macronegativetile{
     \def\south{}
     \def\east{}
     \def\north{}
     \def\west{} 
     \draw (\x + 0, \y) --  (\x+1,\y) node [pos=0.5,above] {\south} -- (\x+1,\y+1) node [pos=0.5,left] {\east} -- (\x+0,\y+1) node [pos=0.5,above] {\north} --  (\x + 0, \y)  node [pos=0.5,left] {\west};
}
\def\matchingcolor{2pt}
 \begin{tikzpicture}[scale=\myscalematching,
 font=\footnotesize
 ]
 \foreach \x in {0,1,...,2} 
 {
\def\y{0}
     \ifthenelse{\x = 0 \OR \x =2}
     {
     \macropositivetile
     }{
     \macronegativetile
     }
     \ifthenelse{\x = 0 \OR \x =2}{\draw[line width=\matchingcolor] (\x + 0, \y+0) -- (\x+1,\y+0);}{} 
     \ifthenelse{\x = 0}{\macrofilledtile \draw[line width=\matchingcolor] (\x + \y+0, \y+1) -- (\x+1,\y+1); }{} 
     \ifthenelse{\x = 0}{\draw (\x + 0.5, \y+0.5) node {\boxed{$a$}}; }{}
     \ifthenelse{\x = 2}{ \macrofilledtile\draw (\x + 0.5, \y+0.5) node {\boxed{$c$}}; }{}
     
          \ifthenelse{\x = 0 \OR \x = 2}{\draw[red,densely dashed,line width=1.5*\matchingcolor] (\x+0, \y+0) -- (\x, \y+1);}{} 
          \ifthenelse{\x = 0 \OR \x = 2}{\draw[red,densely dashed,line width=1.5*\matchingcolor] (\x+1, \y+0) -- (\x+1, \y+1);}{} 
 }
\def\y{1}
 \foreach \x in {2,...,5} 
 {
     \ifthenelse{\x = 2 \OR \x = 3 \OR \x = 5}{\macrofilledtile}{}
     \ifthenelse{\x = 2}
     {\draw (\x + 0.5, \y+0.5) node {{$d$}}; }{}
     \ifthenelse{\x = 3}
     {\draw (\x + 0.5, \y+0.5) node {{$e$}}; }{}
     \ifthenelse{\x = 3 \OR \x =5}
     {\macropositivetile}{\macronegativetile}
     \ifthenelse{\x = 2}{\draw[line width=\matchingcolor] (\x+0, \y+0) -- (\x, \y+1);}{} 
     \ifthenelse{\x = 2}{\draw[red, densely dashed, line width=1.5*\matchingcolor] (\x+0, \y+1) -- (\x+1,\y+1);}{} 
      \ifthenelse{\x = 3}{\draw[red, densely dashed, line width=1.5*\matchingcolor] (\x+1, \y) -- (\x+1,\y+1);}{} 
      \ifthenelse{\x = 3}{\draw[line width=\matchingcolor] (\x+0, \y+0) -- (\x+1, \y+0);}{} 
     \ifthenelse{\x = 3}{\draw[line width=\matchingcolor] (\x+0, \y+1) -- (\x+1,\y+1);}{} 
      \ifthenelse{\x =5}{\draw[line width=\matchingcolor] (\x+0, \y+0) -- (\x+1, \y+0);}{} 
      \ifthenelse{\x = 5}{\draw (\x + 0.5, \y+0.5) node {\boxed{$g$}}; }{}
      \ifthenelse{\x = 5}{\draw[red,densely dashed,line width=1.5*\matchingcolor] (\x+0, \y+0) -- (\x, \y+1);}{} 
          \ifthenelse{\x = 5}{\draw[red,densely dashed,line width=1.5*\matchingcolor] (\x+1, \y+0) -- (\x+1, \y+1);}{} 
 } 
\def\y{2}
 \foreach \x in {5} 
 {
     \macrofilledtile
     \macronegativetile
     \ifthenelse{\x = 5}{\draw[line width=\matchingcolor] (\x+0, \y+0) -- (\x, \y+1);}{} 
     \ifthenelse{\x = 5}{\draw[line width=\matchingcolor] (\x+1, \y) -- (\x+1,\y+1);}{} 
     \draw (\x + 0.5, \y+0.5) node {{$h$}};
 }  
\def\y{3}
\foreach \x in {5,6} 
{
     \ifthenelse{\x = 5}
     {\macropositivetile \macrofilledtile}{\macronegativetile}
     \ifthenelse{\x = 5}{\draw[line width=\matchingcolor] (\x+0, \y+1) -- (\x+1,\y+1);}{} 
     \ifthenelse{\x = 5}{\draw[red,densely dashed,line width=1.5*\matchingcolor] (\x+0, \y+0) -- (\x, \y+1);}{} 
          \ifthenelse{\x = 5}{\draw[red,densely dashed,line width=1.5*\matchingcolor] (\x+1, \y+0) -- (\x+1, \y+1);}{} 

     \ifthenelse{\x = 6}{\draw[line width=\matchingcolor] (\x+1, \y) -- (\x+1,\y+1);}{} 
     \ifthenelse{\x = 5}{\draw (\x + 0.5, \y+0.5) node {\boxed{$i$}}; }{}
 } 
\end{tikzpicture} 
\hspace{-0.8cm}
 \begin{tikzpicture}[scale=\myscalematching,font=\scriptsize]
 \foreach \x in {0,1,...,2} 
 {
\def\y{0}
     \ifthenelse{\x = 0 \OR \x =2}
     {
     \macropositivetile
     }{
     \macronegativetile
     }
     \ifthenelse{\x = 0}{\macrofilledtile 
      }{}
     \ifthenelse{\x = 2}{ \macrofilledtile 
     }{}
     
          \ifthenelse{\x = 0 \OR \x = 2}{\draw[red,line width=1.5*\matchingcolor] (\x+0, \y+0) -- (\x, \y+1);}{} 
          \ifthenelse{\x = 0 \OR \x = 2}{\draw[red,line width=1.5*\matchingcolor] (\x+1, \y+0) -- (\x+1, \y+1);}{} 
 }
\def\y{1}
 \foreach \x in {2,...,5} 
 {
     \ifthenelse{\x = 2 \OR \x = 3 \OR \x = 5}{\macrofilledtile}{}
     \ifthenelse{\x = 3 \OR \x =5}
     {\macropositivetile}{\macronegativetile}
     \ifthenelse{\x = 2}{\draw[red, line width=1.5*\matchingcolor] (\x+0, \y+1) -- (\x+1,\y+1);}{} 
      \ifthenelse{\x = 3}{\draw[red, line width=1.5*\matchingcolor] (\x+1, \y) -- (\x+1,\y+1);}{} 
      \ifthenelse{\x = 5}{\draw[red,line width=1.5*\matchingcolor] (\x+0, \y+0) -- (\x, \y+1);}{} 
          \ifthenelse{\x = 5}{\draw[red,line width=1.5*\matchingcolor] (\x+1, \y+0) -- (\x+1, \y+1);}{} 
 } 
\def\y{2}
 \foreach \x in {5} 
 {
     \macrofilledtile
     \macronegativetile
 }  
\def\y{3}
\foreach \x in {5,6} 
{
     \ifthenelse{\x = 5}
     {\macropositivetile \macrofilledtile}{\macronegativetile}
     \ifthenelse{\x = 5}{\draw[red,line width=1.5*\matchingcolor] (\x+0, \y+0) -- (\x, \y+1);}{} 
          \ifthenelse{\x = 5}{\draw[red,line width=1.5*\matchingcolor] (\x+1, \y+0) -- (\x+1, \y+1);}{} 

     \ifthenelse{\x = 6}{\draw[line width=\matchingcolor] (\x+1, \y) -- (\x+1,\y+1);}{} 
 } 
\end{tikzpicture} 
		\caption{Tiles $\square_L$ associated to blocks $L$ of circled edges (left);
			the set $fil(s)$ of shaded tiles and the set $pm(s)$ of thick solid edges (right) for the subword s=$\textcolor{blue}{\boxed{1}\boxed{1}\boxed{01}\boxed{0}}$
		}
		\label{fig:square_L:s11010}\label{fig:example:pm_acgi}
	\end{center}
\end{figure}

\begin{example}
\label{example:subword_to_pm}
Consider the word $w = 1011101100$. 
In Fig.~\ref{fig:square_L:s101101100},  
we illustrate the following steps. 
\begin{enumerate}[(step 1)]
\item Circle the edges of $G$ corresponding to the subword  $s=\textcolor{blue}{101101100}$. 

\item We get a set of two tiles: $\square_{\textcolor{blue}{1011}}$=$\boxed{d}$ and $\square_{\textcolor{blue}{01100}}$=$\boxed{j}$. 
(Compare $\left\{ \boxed{d}, \boxed{j} \right\}$ with the antichain $A$=$\{\circled{\dd}, \circled{\jj}\}$ from Fig.~\ref{fig:example:antichain_to_subword} (left).) 

\item The two tiles give us two blocks of shaded tiles: 
\[
fil\left(\boxed{d}\right)=\left\{ \boxed{d}, \boxed{e} \right\}
\text{ and }
fil\left(\boxed{j}\right)=\left\{ \boxed{h}, \boxed{i}, \boxed{j} \right\}. 
\]
Their union is 
\[
fil(s)=\left\{\boxed{d}, \boxed{e}, \boxed{h}, \boxed{i}, \boxed{j} \right\}.
\] 
(Compare this 
with the order filter 
$\left\{\circled{\dd}, \circled{\ee}, \circled{\hh}, \circled{\ii}, \circled{\jj} \right\}$
generated by the antichain $A$.)

\item  
The perfect matching $pm(s)$ is the symmetric difference of the edges bounding $fil(s)$ and the edges in the minimal matching of $G$, see the set of thick  edges in Fig.~\ref{fig:square_L:s101101100} (right).
\end{enumerate}

\bigskip

In Fig.~\ref{fig:square_L:s11010}, we illustrate the bijection for a different subword of $w$.
\begin{enumerate}[(step 1)]
\item Circle the edges of $G$ corresponding to the subword  $s=\textcolor{blue}{11010}$.

\item We get a set of four tiles:  
$\square_{\textcolor{blue}{1}}$=$\boxed{a}$, $\square_{\textcolor{blue}{1}}$=$\boxed{c}$ , $\square_{\textcolor{blue}{01}}$=$\boxed{g}$, 
and $\square_{\textcolor{blue}{0}}$=$\boxed{i}$. 
(Compare this with the antichain 
$A$=$\{\circled{\aa}$, 
$\circled{\cc}$, 
$\circled{\gg}$, 
$\circled{\ii}\}$ from Fig.~\ref{fig:example:antichain_to_subword} (right).)

\item 
The four tiles give us three blocks of shaded tiles:  
\begin{align*}
fil\Big(\boxed{a}\Big)&=\Big\{ \boxed{a} \Big\},\\ 
fil\Big(\boxed{c}\Big)&=\left\{ \boxed{c}, 
\boxed{d}, \boxed{e}\right\}, \\
fil\left(\boxed{g}\right)&=\left\{\boxed{g}, \boxed{h}, \boxed{i} \right\}
=fil\left(\boxed{i}\right).
\end{align*}
Their union is 
\[ fil(11010)
=\left\{ \boxed{a}, \boxed{c}, \boxed{d}, \boxed{e}, \boxed{g}, \boxed{h}, \boxed{i}\right\}.  \] 
(Compare this with the order filter 
$\left\{ \circled{\aa}, \circled{\cc}, \circled{\dd}, \circled{\ee}, \circled{\gg}, \circled{\hh}, \circled{\ii}\right\}$
generated by the antichain $A$.) 
\item The perfect matching $pm(s)$ is the symmetric difference of the edges bounding $fil(s)$ and the edges in the minimal matching of $G$, see the set of thick edges in Fig.~\ref{fig:square_L:s11010} (right).
\end{enumerate} 
\end{example}

\begin{remark}
Let $P$ be the piecewise-linear poset and $Q$ the quiver corresponding to a binary word $w$. 
Let $M(w)$ be the indecomposable quiver  representation with dimension vector $(1,1,\dots,1)$. 
Then each order filter $F$ of $P$ corresponds to the support of a subrepresentation $N$ of $M(w)$, see~\cite[Remark~5.5]{MSW13} and \cite{CaSi18}. 
The antichain consisting of the minimal elements of $F$ corresponds to the summands of the projective cover of $N$, see for example~\cite[Section~2.2]{Sch14}.

\def\id{id}
 
For example,
let $Q$ be the following quiver.  
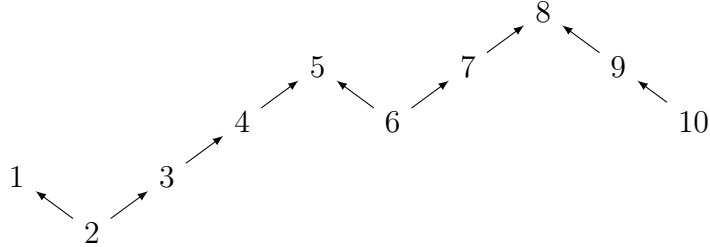
\begin{center}
\begin{tikzpicture}[xscale=\myxscale,yscale=\myyscale,>=latex]
\def\posetedgecolor{blue}
\node(1) at (0,0) {{$1$}}; 
\node(2) at (1,-1) {$2$}; 
\node(3) at (2,0) {{$3$}}; 
\node(4) at (3,1) {$4$}; 
\node(5) at (4,2) {$5$}; 
\node(6) at (5,1) {$6$}; 
\node(7) at (6,2) {{$7$}}; 
\node(8) at (7,3) {$8$};
\node(9) at (8,2) {{$9$}};
\node(10) at (9,1)  {$10$};

\draw[->] (2) -- (1); 
\draw[->]  (2) -- (3); 
\draw[->] (3) -- (4); 
\draw[->] (4) -- (5); 
\draw[->] (6) -- (5); 
\draw[->] (6) -- (7); 
\draw[->] (7) -- (8); 
\draw[->] (9) -- (8); 
\draw[->] (10) -- (9); 
\end{tikzpicture}
\end{center}

 The first order filter in 
Example~\ref{example:subword_to_pm} 
corresponds to the quiver representation

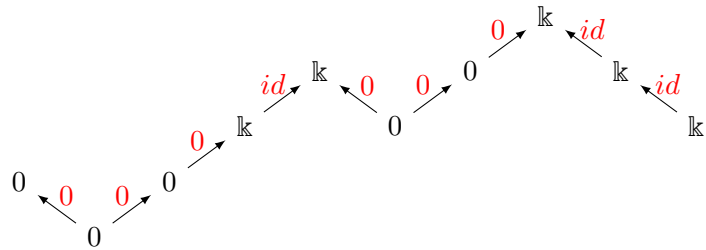
\begin{center}
\begin{tikzpicture}[yscale=\myyscale, xscale=\myxscale,
>=latex, 
font = \small
]
\def\posetedgecolor{red}
\node(1) at (0,0) {$0$}; 
\node(2) at (1,-1) {$0$}; 
\node(3) at (2,0) {$0$}; 
\node(4) at (3,1) {$\Bbbk$}; 
\node(5) at (4,2) {$\Bbbk$}; 
\node(6) at (5,1) {$0$}; 
\node(7) at (6,2) {$0$}; 
\node(8) at (7,3) {$\Bbbk$}; 
\node(9) at (8,2) {$\Bbbk$}; 
\node(10) at (9,1) {$\Bbbk$}; 

\draw[->] (2) -- (1) node[\posetedgecolor,pos=0.25,above] {$0$}; 
\draw[->]  (2) -- (3) node[\posetedgecolor,pos=0.25,above] {$0$}; 
\draw[->] (3) -- (4) node[\posetedgecolor,pos=0.25,above] {$0$};  
\draw[->] (4) -- (5) node[\posetedgecolor,pos=0.25,above] {$\id$}; 
\draw[->] (6) -- (5) node[\posetedgecolor,pos=0.25,above] {$0$}; 
\draw[->] (6) -- (7) node[\posetedgecolor,pos=0.25,above] {$0$};  
\draw[->] (7) -- (8) node[\posetedgecolor,pos=0.25,above] {$0$};  
\draw[->] (9) -- (8) node[\posetedgecolor,pos=0.25,above] {$\id$};  
\draw[->] (10) -- (9) node[\posetedgecolor,pos=0.25,above] {$\id$}; 
\end{tikzpicture}
\end{center}
of $Q$. 
Its projective cover is the direct sum $P(4)\oplus P(10)$ of $P(4)$, the projective representation at vertex $4$, and $P(10)$, the projective representation at vertex $10$. 

The second order filter in Example~\ref{example:subword_to_pm} 
corresponds to the quiver representation 
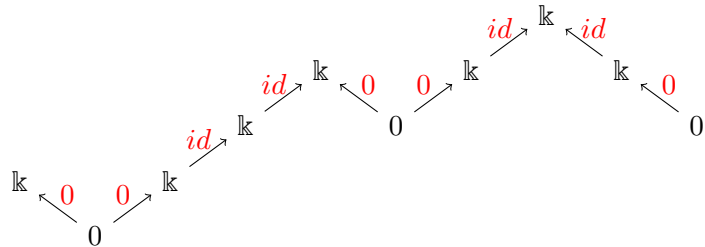
\begin{center}
\begin{tikzpicture}[yscale=\myyscale, xscale=\myxscale, 
font = \small 
]
\def\posetedgecolor{red}
\node(1) at (0,0) {$\Bbbk$}; 
\node(2) at (1,-1) {$0$}; 
\node(3) at (2,0) {$\Bbbk$}; 
\node(4) at (3,1) {$\Bbbk$}; 
\node(5) at (4,2) {$\Bbbk$}; 
\node(6) at (5,1) {$0$}; 
\node(7) at (6,2) {$\Bbbk$}; 
\node(8) at (7,3) {$\Bbbk$}; 
\node(9) at (8,2) {$\Bbbk$}; 
\node(10) at (9,1) {$0$}; 

\draw[->] (2) -- (1) node[\posetedgecolor,pos=0.25,above] {$0$}; 
\draw[->] (2) -- (3) node[\posetedgecolor,pos=0.25,above] {$0$}; 
\draw[->] (3) -- (4) node[\posetedgecolor,pos=0.25,above] {$\id$};  
\draw[->] (4) -- (5) node[\posetedgecolor,pos=0.25,above] {$\id$}; 
\draw[->] (6) -- (5) node[\posetedgecolor,pos=0.25,above] {$0$}; 
\draw[->] (6) -- (7) node[\posetedgecolor,pos=0.25,above] {$0$};  
\draw[->] (7) -- (8) node[\posetedgecolor,pos=0.25,above] {$\id$};  
\draw[->] (9) -- (8) node[\posetedgecolor,pos=0.25,above] {$\id$};  
\draw[->] (10) -- (9) node[\posetedgecolor,pos=0.25,above] {$0$}; 
\end{tikzpicture}
\end{center} 
of $Q$. 
Its projective cover is the direct sum $P(1)\oplus P(3) \oplus P(7) \oplus P(9)$ 
of the projective representations at vertices $1$, $3$, $7$, and $9$, respectively.

Both representations are subrepresentations of the indecomposable quiver representation $M(w)$ 
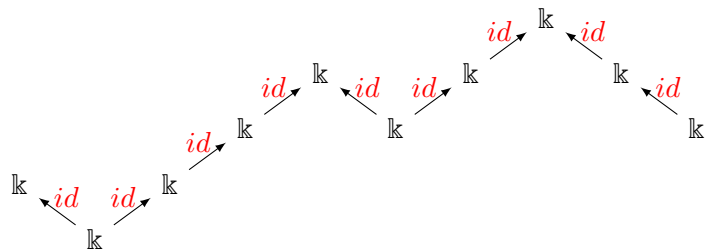
\begin{center}
\begin{tikzpicture}[yscale=\myyscale, xscale=\myxscale, 
>=latex, 
font = \small
]
\def\posetedgecolor{red}
\node(1) at (0,0) {$\Bbbk$}; 
\node(2) at (1,-1) {$\Bbbk$}; 
\node(3) at (2,0) {$\Bbbk$}; 
\node(4) at (3,1) {$\Bbbk$}; 
\node(5) at (4,2) {$\Bbbk$}; 
\node(6) at (5,1) {$\Bbbk$}; 
\node(7) at (6,2) {$\Bbbk$}; 
\node(8) at (7,3) {$\Bbbk$}; 
\node(9) at (8,2) {$\Bbbk$}; 
\node(10) at (9,1) {$\Bbbk$}; 

\draw[->] (2) -- (1) node[\posetedgecolor,pos=0.25,above] {$\id$}; 
\draw[->]  (2) -- (3) node[\posetedgecolor,pos=0.25,above] {$\id$}; 
\draw[->] (3) -- (4) node[\posetedgecolor,pos=0.25,above] {$\id$};  
\draw[->] (4) -- (5) node[\posetedgecolor,pos=0.25,above] {$\id$}; 
\draw[->] (6) -- (5) node[\posetedgecolor,pos=0.25,above] {$\id$}; 
\draw[->] (6) -- (7) node[\posetedgecolor,pos=0.25,above] {$\id$};  
\draw[->] (7) -- (8) node[\posetedgecolor,pos=0.25,above] {$\id$};  
\draw[->] (9) -- (8) node[\posetedgecolor,pos=0.25,above] {$\id$};  
\draw[->] (10) -- (9) node[\posetedgecolor,pos=0.25,above] {$\id$}; 
\end{tikzpicture}
\end{center}
of $Q$.
\end{remark}

\subsection*{Acknowledgements}
This project was inspired by conversations with M.~Stipulanti and L.~Tarsissi during the conference ``Sage Days 82 : Women in Sage", held January 2017 in Paris, funded by the OpenDreamKit project. We thank the organizers (J.~Balakrishnan, V.~Pons, and 
J.~Striker) for creating such a productive environment. 
We are also grateful for helpful conversations with E.~Barnard, R.~Schiffler, K.~Serhiyenko, and E.~Y{\i}ld{\i}r{\i}m. 
Finally, we thank the anonymous reviewers whose suggestions helped improve and clarify this paper.
Most of this work was completed while the second author was an Assistant Research Professor at University of Connecticut.

\bibliography{bib}
\bibliographystyle{alpha}

\end{document}